\documentclass{amsart}%
\usepackage{amsfonts}
\usepackage{hyperref}
\usepackage{adjustbox}
\usepackage[table]{xcolor}
\usepackage{tikz-cd}
\usepackage[all,cmtip]{xy}
\usepackage[margin=1.51in]{geometry}
\usepackage{amsmath}
\usepackage{amssymb}
\usepackage{graphicx}%
\providecommand{\U}[1]{\protect\rule{.1in}{.1in}}

\newtheorem{theorem}{Theorem}
\numberwithin{theorem}{section}
\theoremstyle{plain}
\newtheorem*{acknowledgement}{Acknowledgement}

\newtheorem{corollary}[theorem]{Corollary}

\newtheorem{definition}[theorem]{Definition}

\newtheorem{lemma}[theorem]{Lemma}

\newtheorem{proposition}[theorem]{Proposition}
\theoremstyle{remark}
\newtheorem{remark}[theorem]{Remark}
\newtheorem{caution}[theorem]{Caution}

\newtheorem{example}[theorem]{Example}

\numberwithin{equation}{section}

\definecolor{lg}{rgb}{0.8,0.8,0.8}
\begin{document}
\title[ ]{Without real vector spaces all regulators are rational}
\author[O. Braunling]{Oliver Braunling}
\address{Dortmund University of Applied Sciences, Emil-Figge-Stra\ss e 42, 44227
Dortmund, Germany}
\thanks{The author acknowledges support for this article as part of Grant CNS2023-145167
funded by MICIU/AEI/10.13039/501100011033.}

\begin{abstract}
Every LCA group has a Haar measure unique up to rescaling by a positive
scalar. Clausen has shown that the Haar measure describes the universal
determinant functor of the category LCA in the sense of Deligne. We show that
when only working with LCA groups without allowing real vector spaces, any
conceivable determinant functor is unique up to rescaling by at worst
\textit{rational} values. As a result, no transcendental real nor $p$-adic
regulators could ever show up in special $L$-value conjectures (as in Tamagawa
number conjectures or Weil-\'{e}tale cohomology) if anyone had the, admittedly
outlandish and bizarre, idea to try to circumvent incorporating a real (Betti)
realization of the motive.

\end{abstract}
\maketitle

\section{Overview}

The main result of this note will not shock anyone: It is hard to come by a
transcendental real number without using real numbers. Let that sink in.
However, the result we show is more precise and we mean different things: Real
numbers as an object in the category of LCA groups, versus real numbers
showing up as regulators. The standard conjectures on special $L$-values
intertwine arithmetic cohomological values with transcendental regulator
values. Dirichlet's analytic class number formula or the B-SD conjecture are
probably the most famous examples. Conjecturally, in the picture devised by
Deligne and Beilinson, the regulators occur from determinant mismatches
resulting from the comparison of various determinant lines $\Lambda^{\max
}H^{\bullet}(-)$ defined through groups,

- of $p$-adic type, coming from $p$-adic realizations for all primes $p$,

- of real type, coming from the Betti realization,

- and of integral or rational type, coming from motivic cohomology.

There are various formulations of conjectures for this picture, each a bit
different:\ For example, the Bloch--Kato picture of motivic Tamagawa numbers
\cite{MR1086888,MR1265546}, or Lichtenbaum's picture based on Weil-\'{e}tale
cohomology \cite{MR2552104,MR4699877}. We are mostly inspired by recent work
in the direction of Weil-\'{e}tale cohomology theories with coefficients in
LCA groups and output as LCA groups. This has recently been featured
prominently in the work of Flach--Morin \cite{MR3874942} and Geisser--Morin
\cite{MR4699875}. As well as in the work of Artusa \cite{MR4831262,artusa2025}%
, but also in older works like Kottwitz--Shelstad \cite[Appendix E]%
{MR1687096}, Oesterl\'{e} \cite{MR750319}.

However, to show our claim, we will not even touch motives nor Weil-\'{e}tale
cohomology anywhere. We just show that no transcendental (real) values can
occur from determinant functors when working on the category of LCA groups,
\textit{but disallowing real vector space summands}. This is possible because
the determinant lines in the various forms of special $L$-value conjectures
all come from determinant functors in the sense of Deligne. We can prove that
\textit{any} determinant functor on LCA groups $-$ as long as no real vector
spaces ever show up $-$ must factor through \textit{rational} numbers.

Write $\mathsf{LCA}$ for the quasi-abelian category of locally compact abelian
(LCA) groups. Write $\mathsf{LCA}_{\operatorname*{vf}}\subset\mathsf{LCA}$ for
the full subcategory of those groups not having a real line $\mathbb{R}$ as a
direct summand (in group theory, such groups are called `\textit{vector-free}%
', whence the subscript originates). This is an exact category, see
\S \ref{sect_LCA}. Real quotients, like $\mathbb{R}^{n}/\Lambda$ for a full
rank lattice $\Lambda$, remain allowed.

\begin{theorem}
\label{thm_RationalityOfHaarTorsor}The restriction of the Haar functor%
\[
Ha\colon\mathsf{LCA}^{\times}\rightarrow\mathsf{Tors}(\mathbb{R}_{>0}^{\times
})
\]
to $\mathsf{LCA}_{\operatorname*{vf}}$ only attains rational values:%
\[
Ha^{\mathbb{Q}}\colon\mathsf{LCA}_{\operatorname*{vf}}^{\times}\rightarrow
\mathsf{Tors}(\mathbb{Q}_{>0}^{\times})\text{.}%
\]
This functor $Ha^{\mathbb{Q}}$ is the universal determinant functor of
$\mathsf{LCA}_{\operatorname*{vf}}$, i.e., for \emph{any} determinant functor
$\mathcal{D}\colon\mathsf{LCA}_{\operatorname*{vf}}^{\times}\rightarrow
\mathsf{P}$ there exists a factorization%
\[%
{
\begin{tikzcd}
	{\mathsf{LCA}_{\operatorname{vf} }^{\times}} && {(\mathsf{Tors}(\mathbb
{Q}_{>0}^{\times}),\otimes)} \\
	\\
	&& {(\mathsf{P},\boxtimes)}
	\arrow["{{Ha^\mathbb{Q} }}", from=1-1, to=1-3]
	\arrow["{{\mathcal{D}}}"', from=1-1, to=3-3]
	\arrow["f", dashed, from=1-3, to=3-3]
\end{tikzcd}
}%
\]
with $f$ a symmetric monoidal functor of Picard groupoids.
\end{theorem}

See Theorem \ref{thm_main} for a precise statement and the proof, but let us
illustrate an easy consequence:\ Suppose you are given any connected diagram
of objects in $\mathsf{LCA}$ and all arrows are isomorphisms. Then fixing a
Haar measure on any object and pushing it forward along the arrows to any
other object, we may get several \textit{distinct} normalizations%
\[%
{
\adjustbox{scale=0.42}{
\begin{tikzcd}
	&&&& {T'} \\
	& {Z'} &&&& T & {J'} \\
	{V'} && Z &&&&& J \\
	V && U && {Y'} && {X'} \\
	&&&&& Y && X
	\arrow[from=1-5, to=2-6]
	\arrow[from=1-5, to=2-7]
	\arrow[from=1-5, to=4-5]
	\arrow[from=2-2, to=1-5]
	\arrow[from=2-2, to=3-3]
	\arrow[from=2-2, to=4-3]
	\arrow[from=2-6, to=3-8]
	\arrow[from=2-6, to=5-6]
	\arrow[from=2-7, to=3-8]
	\arrow[from=2-7, to=4-7]
	\arrow[from=3-1, to=2-2]
	\arrow[from=3-1, to=3-3]
	\arrow[from=3-3, to=2-6]
	\arrow[from=3-3, to=4-3]
	\arrow[from=3-3, to=5-8]
	\arrow[from=3-8, to=5-8]
	\arrow[from=4-1, to=2-2]
	\arrow[from=4-1, to=3-1]
	\arrow[from=4-1, to=3-3]
	\arrow[from=4-1, to=5-6]
	\arrow[from=4-3, to=4-5]
	\arrow[from=4-5, to=4-7]
	\arrow[from=4-5, to=5-6]
	\arrow[from=4-7, to=5-8]
	\arrow[from=5-8, to=5-6]
\end{tikzcd}}
}%
\]
depending on what path we follow. The first point is: If all objects are in
$\mathsf{LCA}_{\operatorname*{vf}}$, then the normalization can only differ by
rational numbers $-$ irrespective of the shape of the diagram. Secondly (and
this is the true content of the theorem):\ \textit{Any} determinant functor in
the sense of Deligne \cite{MR902592} (or see Def. \ref{def_DeterminantFunctor}
below) must have the same property. Of course, you may still decide to
normalize the Haar measure on, say $\mathbb{Q}_{p}^{n}$, such that%
\[
\operatorname*{vol}\left(  \mathbb{Z}_{p}^{n}\right)  :=\sqrt{2\pi}\text{,
}\zeta(3)\text{ or }e\text{ etc.}%
\]
if you so please, so there would suddenly a transcendental number be involved
$-$ but then any pushforwards of this measure along various paths in any
connected diagram can only differ by rational factors.

To rule out a possible misunderstanding: If we consider the diagram%
\begin{equation}%
{
\begin{tikzcd}
	{\mathbb{Q}_{p}} && {\mathbb{Q}_{p}}
	\arrow["{\cdot1}", shift left=2, from=1-1, to=1-3]
	\arrow["{\cdot\operatorname{log}_p(*)}"', shift right=2, from=1-1, to=1-3]
\end{tikzcd}
}
\label{l_h_1}%
\end{equation}
with some transcendental logarithm value, then clearly the pushforwards of the
(trivial)\ determinant line on the left side along either arrow will differ by
$\log_{p}(\ast)$, which need not be rational. But the point is: The $p$-adic
determinant line does not extend to a determinant functor on $\mathsf{LCA}$ or
$\mathsf{LCA}_{\operatorname*{vf}}$. It is impossible to systematically extend
it to respect exact sequences like%
\[
\mathbb{Z}_{p}\hookrightarrow\mathbb{Q}_{p}\twoheadrightarrow\mathbb{Q}%
_{p}/\mathbb{Z}_{p}\text{,}%
\]
where it is unclear\footnote{and as one can show: impossible} how to attach a
line to $\mathbb{Q}_{p}/\mathbb{Z}_{p}$ such that all axioms of a determinant
functor remain in place. The Haar measure, however, does indeed prolong to all
of $\mathsf{LCA}$, but it only sees the valuation of $p$-adic numbers. So,
writing $\log_{p}(\ast)$ as a power $p^{r}u$ with $u\in\mathbb{Z}_{p}^{\times
}$ a $p$-adic unit, one can check that the ratio of volumes in Eq. \ref{l_h_1}
will be $p^{-r}$, which \textit{is} rational.\medskip

\textit{Strategy of the proof:} Our result is basically just a computation in
$K$-theory. By an old idea of Deligne, every exact category has a universal
determinant functor which all determinants must factor through. This stems
from \cite{MR902592}. He also showed that this universal determinant is
completely determined by the specifics of the $K$-theory spectrum of the
category, truncated to degrees $[0,1]$. This is a tiny bit more information
than just knowing $K_{0}$ and $K_{1}$, and also involves the glueing data of
the Postnikov tower of these two layers (the stable $k$-invariant). It turns
out that it is not too hard to compute these invariants. While this argument
exhibits the group $\mathbb{Q}_{>0}^{\times}$ abstractly, the tricky part is
to link this up to the computation of the rescaling factor under the Haar
measure.\medskip

\textit{Conventions:} In this text, $K$ denotes non-connective $K$-theory,
$K^{\operatorname*{conn}}$ denotes connective (Quillen) $K$-theory,
$\mathcal{U}^{\operatorname*{loc}}$ denotes the localizing non-commutative
motive, $\left.  _{n}A\right.  :=\{a\in A\mid a^{n}=1\}$ denotes the elements
in an abelian group killed by $n$.

\section{LCA groups without real line summands\label{sect_LCA}}

Let $\mathsf{LCA}$ (resp. $\mathsf{LCA}_{\operatorname*{vf}}$) be the category
whose objects are locally compact Hausdorff topological abelian groups (resp.
without a real line direct summand) and morphisms are continuous group
homomorphisms. This category has all kernels and cokernels and is
quasi-abelian. Its natural exact structure is such that%
\[
G^{\prime}\hookrightarrow G\twoheadrightarrow G^{\prime\prime}%
\]
is exact if the first arrow is an injective closed map (these are the
admissible monics) and the second arrow is a surjective closed map (these are
the admissible epics) and the underlying sequence of abelian groups is exact.

\begin{proposition}
[Structure theorem for LCA]\label{prop_struct}Every group $G\in\mathsf{LCA}$
is (non-canonically) isomorphic to $G\simeq G_{0}\oplus\mathbb{R}^{n}$ for
some $n<\infty$ and $G_{0}$ has a (non-unique)\ clopen compact subgroup $C$,
i.e., there is an exact sequence%
\[
C\oplus\mathbb{R}^{n}\hookrightarrow G\twoheadrightarrow D
\]
with $D$ discrete.
\end{proposition}

See for example \cite[Theorem 14.2.18]{MR4510389}. We immediately obtain that
for $G\in\mathsf{LCA}_{\operatorname*{vf}}$ the same is true, but we
additionally know that $n=0$.

\begin{caution}
[{\cite[\S 10.5]{MR2606234}}]Note that $\mathsf{LCA}_{\operatorname*{vf}}$ is
not a fully exact subcategory of $\mathsf{LCA}$, as is witnessed by the exact
sequence $\mathbb{Z}\hookrightarrow\mathbb{R}\twoheadrightarrow\mathbb{T}$
($\mathbb{T}$ the circle group), which shows that $\mathsf{LCA}%
_{\operatorname*{vf}}$ is not closed under extensions inside $\mathsf{LCA}$.
Correspondingly, we cannot expect $\operatorname*{D}\nolimits_{\infty}%
^{b}(\mathsf{LCA}_{\operatorname*{vf}})\rightarrow\operatorname*{D}%
\nolimits_{\infty}^{b}(\mathsf{LCA})$ to be fully faithful.
\end{caution}

\begin{example}
The group $\mathbb{R}_{d}$ of real numbers, but with the discrete topology, is
still allowed in $\mathsf{LCA}_{\operatorname*{vf}}$. However, all its Haar
measures are rescalings of the counting measure and the measure is only finite
on finite subsets. A comparison of volumes like $[-1,+1]$ vs. $[-\frac{1}%
{2},+\frac{1}{2}]$ as in $\mathbb{R}$ is impossible in $\mathbb{R}_{d}$.
\end{example}

\section{Determinant functors}

\begin{definition}
\label{def_PicardGroupoid}A \emph{Picard groupoid} $(\mathsf{P},\otimes)$ is a
($1$-categorical) groupoid $\mathsf{P}$, equipped with a unital symmetric
monoidal structure%
\[
\otimes\colon\mathsf{P}\times\mathsf{P}\longrightarrow\mathsf{P}%
\]
such that all objects are $\otimes$-invertible. We write $1_{\mathsf{P}}$ for
the neutral element of the $\otimes$-structure.\footnote{Being $\otimes
$\emph{-invertible} means that for every $X\in\mathsf{P}$, there exists some
$X^{-1}\in\mathsf{P}$ such that $X^{-1}\otimes X\simeq1_{\mathsf{P}}$ (with no
requirements on naturality). There exist various variants of this definition,
the strongest being the existence of a functor of inversion $(-)^{-1}%
\colon(\mathsf{P},\otimes)\longrightarrow(\mathsf{P},\otimes^{op})$.
Essentially, one can show that merely assuming the existence of $\otimes
$-inverses, it is always possible to extend this to a functor of inversion.
However, we shall not need to know any of this.}
\end{definition}

Write $\mathsf{Picard}$ for the $2$-category of Picard groupoids with
symmetric monoidal functors as $1$-arrows and natural equivalences as $2$-arrows.

We write $\pi_{0}(\mathsf{P},\otimes)$ for the group of isomorphism classes in
$\mathsf{P}$ with $\otimes$ as its multiplication. For every object
$X\in\mathsf{P}$, the self-symmetry $s_{X,X}\colon X\otimes X\overset{\sim
}{\longrightarrow}X\otimes X$ must satisfy $s_{X,X}\circ s_{X,X}%
=\operatorname*{id}_{X\otimes X}$ by the symmetry axiom of a symmetric
monoidal category. Multiplying from the right with the identity $(X\otimes
X)^{-1}\overset{\sim}{\longrightarrow}(X\otimes X)^{-1}$ induces an
isomorphism%
\begin{equation}
s_{X,X}\otimes X^{-1}\otimes X^{-1}\colon1_{\mathsf{P}}\overset{\sim
}{\longrightarrow}1_{\mathsf{P}}\text{,} \label{lwup6a}%
\end{equation}
known as the \emph{signature} $\varepsilon_{X}$ \cite[\S 2.5.3]{MR0338002}.
One finds that%
\[
\varepsilon_{X\otimes Y}=\varepsilon_{X}\cdot\varepsilon_{Y}\text{.}%
\]
This induces a well-defined group homomorphism%
\[
\varepsilon\colon\pi_{0}(\mathsf{P},\otimes)\longrightarrow\left.  _{2}\pi
_{1}(\mathsf{P},\otimes)\right.
\]
to the $2$-torsion elements of $\pi_{1}(\mathsf{P},\otimes
):=\operatorname*{Aut}_{\mathsf{P}}(1_{\mathsf{P}})$.

\begin{definition}
\label{def_Signature}The morphism $\varepsilon$ is the \emph{stable }%
$k$\emph{-invariant} of $(\mathsf{P},\otimes)$.
\end{definition}

Let $A$ be an abelian group.

\begin{definition}
Write $\mathsf{Tors}(A)$ for the groupoid of $A$-torsors\footnote{if a torsor
is a sheaf for you, regard it as a sheaf on an unnamed one-point set with the
trivial topology.}:

\begin{enumerate}
\item Objects are left $A$-torsors, i.e., a set $X$ with a free simply
transitive left $A$-action%
\[
A\times X\longrightarrow X\text{,}%
\]

\item and morphisms are bijections $\phi$ of $A$-sets, preserving the left
$A$-action%
\[
\phi(a\cdot x)=a\cdot\phi(x)\text{.}%
\]
There is a designated object, the trivial torsor:%
\[
1_{\mathsf{Tors}(A)}:=A
\]
with its natural left action given by multiplication in $A$.
\end{enumerate}
\end{definition}

The automorphisms of $1_{\mathsf{Tors}(A)}$ in this category are canonically
isomorphic to $A$ itself, $\operatorname*{Aut}(1_{\mathsf{Tors}(A)})\cong A$.
The groupoid $\mathsf{Tors}(A)$ can be promoted to become a Picard groupoid:

\begin{definition}
Write $(\mathsf{Tors}(A),\otimes)$ for the Picard groupoid of $A$-torsors:

\begin{enumerate}
\item The monoidal multiplication is%
\[
X\otimes Y:=(X\times Y)/\sim\text{,}%
\]
where $(x,y):\sim(ax,a^{-1}y)$ for all $a\in A$. Equip this set with the left
$A$-action $a\cdot(x,y):=(ax,y)=(x,ay)$.

\item The inverse object $X^{-1}$ is defined to be the same set $X$, but with
the left action $a\cdot_{X^{-1}}x:=a^{-1}\cdot x$. Then%
\[
X\otimes X^{-1}\overset{\sim}{\longrightarrow}1_{\mathsf{Tors}(A)}%
\]
sending $x^{\prime}\otimes x$ to the unique element $a\in A$ such that $a\cdot
x=x^{\prime}$.

\item The associativity and symmetry constraint are trivial, e.g.,%
\begin{align*}
s_{X,Y}\colon Y\otimes X  &  \longrightarrow X\otimes Y\\
(y,x)  &  \longmapsto(x,y)\text{.}%
\end{align*}

\end{enumerate}
\end{definition}

Since all objects in this category are pairwise isomorphic, we have $\pi
_{0}(\mathsf{Tors}(A),\otimes)=0$ and $\pi_{1}(\mathsf{Tors}(A),\otimes)\cong
A$.

If $\psi\colon A\rightarrow B$ is a homomorphism of abelian groups, there is
an induced (symmetric monoidal) basechange functor of Picard groupoids%
\begin{align}
\psi_{\ast}\colon(\mathsf{Tors}(A),\otimes)  &  \longrightarrow(\mathsf{Tors}%
(B),\otimes)\label{lrio0}\\
X  &  \longmapsto(B\times X)/\sim\qquad\text{(some people write }B\otimes
_{A}X\text{)}\nonumber
\end{align}
with $(x,y):\sim(\psi(a)x,a^{-1}y)$ for all $a\in A$, $x\in B$ and $y\in X$.
Equip this set with the left $B$-action $b\cdot(x,y):=(bx,y)$.

Let $\mathsf{C}$ be an exact category. Let $\mathsf{C}^{\times}$ denote its
maximal inner groupoid, i.e., the category with the same objects, but we only
keep isomorphisms as morphisms.

\begin{definition}
[{\cite[\S 4.3]{MR902592}}]\label{def_DeterminantFunctor}Let $\mathsf{C}$ be
an exact category and let $(\mathsf{P},\otimes)$ be a Picard groupoid. A
\emph{determinant functor} on $\mathsf{C}$ is a functor%
\[
\mathcal{D}\colon\mathsf{C}^{\times}\longrightarrow\mathsf{P}%
\]
along with the following extra structure and axioms:

\begin{enumerate}
\item For any exact sequence $\Sigma\colon G^{\prime}\hookrightarrow
G\twoheadrightarrow G^{\prime\prime}$ in $\mathsf{C}$, we are given an
isomorphism%
\begin{equation}
\mathcal{D}(\Sigma)\colon\mathcal{D}(G)\overset{\sim}{\longrightarrow
}\mathcal{D}(G^{\prime})\underset{\mathsf{P}}{\otimes}\mathcal{D}%
(G^{\prime\prime}) \label{lrio1}%
\end{equation}
in $\mathsf{P}$. This isomorphism is required to be functorial in morphisms of
exact sequences.

\item For every zero object $Z$ of $\mathsf{C}$, we are provided with an
isomorphism $z\colon\mathcal{D}(Z)\overset{\sim}{\rightarrow}1_{\mathsf{P}}$
to the neutral object of the Picard groupoid.

\item Suppose $f\colon G\rightarrow G^{\prime}$ is an isomorphism in
$\mathsf{C}$. We write%
\[
\Sigma_{l}\colon0\hookrightarrow G\twoheadrightarrow G^{\prime}\qquad
\text{and}\qquad\Sigma_{r}\colon G\hookrightarrow G^{\prime}\twoheadrightarrow
0
\]
for the depicted exact sequences. We demand that the composition%
\begin{equation}
\mathcal{D}(G)\underset{\mathcal{D}(\Sigma_{l})}{\overset{\sim
}{\longrightarrow}}\mathcal{D}(0)\underset{\mathsf{P}}{\otimes}\mathcal{D}%
(G^{\prime})\underset{z\otimes1}{\overset{\sim}{\longrightarrow}}%
1_{\mathsf{P}}\underset{\mathsf{P}}{\otimes}\mathcal{D}(G^{\prime
})\underset{g_{\mathcal{D}(G^{\prime})}}{\overset{\sim}{\longleftarrow}%
}\mathcal{D}(G^{\prime}) \label{l_CDetFunc1}%
\end{equation}
and the natural map $\mathcal{D}(f)\colon\mathcal{D}(G)\overset{\sim
}{\rightarrow}\mathcal{D}(G^{\prime})$ agree. We further require that
$\mathcal{D}(f^{-1})$ agrees with a variant of Equation \ref{l_CDetFunc1}
using $\Sigma_{r}$ instead of $\Sigma_{l}$.

\item If a two-step filtration $G_{1}\hookrightarrow G_{2}\hookrightarrow
G_{3}$ is given, we demand that the diagram%
\begin{equation}%
\xymatrix{
\mathcal{D}(G_3) \ar[r]^-{\sim} \ar[d]_{\sim} & \mathcal{D}(G_1) \otimes
\mathcal{D}(G_3/G_1) \ar[d]^{\sim} \\
\mathcal{D}(G_2) \otimes\mathcal{D}(G_3/G_2) \ar[r]_-{\sim} & \mathcal
{D}(G_1) \otimes\mathcal{D}(G_2/G_1) \otimes\mathcal{D}(G_3/G_2)
}
\label{l_CDetA}%
\end{equation}
commutes.

\item Given objects $G,G^{\prime}\in\mathsf{C}$ consider the exact sequences%
\[
\Sigma_{1}\colon G\hookrightarrow G\oplus G^{\prime}\twoheadrightarrow
G^{\prime}\qquad\text{and}\qquad\Sigma_{2}\colon G^{\prime}\hookrightarrow
G\oplus G^{\prime}\twoheadrightarrow G
\]
with the natural inclusion and projection morphisms. Then the diagram%
\[%
\xymatrix{
& \mathcal{D}(G \oplus G^{\prime}) \ar[dl]_{\mathcal{D}(\Sigma_1)}
\ar[dr]^{\mathcal{D}(\Sigma_2)} \\
\mathcal{D}(G) \otimes\mathcal{D}(G^{\prime}) \ar[rr]_{s_{G,G^{\prime}}}
& & \mathcal{D}(G^{\prime}) \otimes\mathcal{D}(G)
}%
\]
commutes, where $s_{G,G^{\prime}}$ denotes the symmetry constraint of
$\mathsf{P}$.
\end{enumerate}
\end{definition}

At the end of \cite[\S 4.3]{MR902592}, Deligne considers the \emph{category of
determinant functors} $\det(\mathsf{C},\mathsf{P})$:

\begin{enumerate}
\item objects are determinant functors in the sense of the above definition, and

\item morphisms are natural transformations of determinant functors.
\end{enumerate}

Details can be found spelled out in \cite[\S 2.3]{MR2842932}, especially a
full description of a morphism of determinant functors is \cite[Definition
2.5]{MR2842932}. We also took over his notation $\det(\mathsf{C},\mathsf{P})$
for this category.

\begin{definition}
\label{def_UnivDetFunctor}A determinant functor $\mathcal{D\colon}%
\mathsf{C}^{\times}\longrightarrow\mathsf{P}$ is called \emph{universal} if
for every given Picard groupoid $\mathsf{P}^{\prime}$ the functor%
\[
\operatorname*{Hom}\nolimits^{\otimes}(\mathsf{P},\mathsf{P}^{\prime
})\longrightarrow\det\left(  \mathsf{C},\mathsf{P}^{\prime}\right)
\,\text{,}\qquad\varphi\mapsto\varphi\circ\mathcal{D}%
\]
is an equivalence of categories.
\end{definition}

This is in Deligne \cite[\S 4.3]{MR902592}, but perhaps a little more detailed
in \cite[\S 4.1]{MR2842932}.

For an LCA group $G$, we write $C_{c}(G,\mathbb{R})$ to denote the (possibly
non-unital) Banach algebra of continuous real-valued functions with compact support.

The following seems to be known among all specialists, but we are not aware of
any detailed account in the literature:

\begin{definition}
Write%
\[
Ha\colon\mathsf{LCA}^{\times}\longrightarrow\mathsf{Tors}(\mathbb{R}%
_{>0}^{\times})
\]
for the determinant functor sending an LCA group to its set of Haar measures.

\begin{enumerate}
\item This is a left $\mathbb{R}_{>0}^{\times}$-torsor by multiplying the Haar
measure with a positive real scalar.

\item For any exact sequence $\Sigma\colon G^{\prime}\hookrightarrow
G\twoheadrightarrow G^{\prime\prime}$ in $\mathsf{LCA}$, the isomorphism%
\begin{equation}
Ha(\Sigma)\colon Ha(G)\overset{\sim}{\longrightarrow}Ha(G^{\prime})\otimes
Ha(G^{\prime\prime}) \label{lg_1}%
\end{equation}
in $\mathsf{Tors}(\mathbb{R}_{>0}^{\times})$ is the inverse of the following
construction: Given Haar measures $\mu_{G^{\prime}}$ and $\mu_{G^{\prime
\prime}}$ on the closed subgroup $G^{\prime}$, and quotient group
$G^{\prime\prime}$, there is a unique normalization of the Haar measure on $G$
such that%
\begin{equation}
\int_{G^{\prime\prime}}\int_{G^{\prime}}f(x\xi)d\mu_{G^{\prime}}(\xi
)d\mu_{G^{\prime\prime}}(\overline{x})=\int_{G}f(x)d\mu(x) \label{lg_2}%
\end{equation}
holds for all $f\in C_{c}(G,\mathbb{R})$.

\item For the zero group $\{0\}$, $z\colon\mathcal{D}(\{0\})\overset{\sim
}{\rightarrow}1_{\mathsf{P}}$ is the choice of multiples of the counting
measure $\mu_{\{0\}}(\{0\})=r\cdot1$ for $r\in\mathbb{R}_{>0}^{\times}$.
\end{enumerate}
\end{definition}

We supply some details: (1) The existence and uniqueness up to positive
scalars of Haar measures can for example be found in \cite[Theorem 2.10,
Theorem 2.20]{MR3444405} or \cite[Theorem 29C and D]{MR54173}. This defines
$Ha$ on objects. Given an arrow $G^{\prime}\overset{F}{\longrightarrow}G$ in
$\mathsf{LCA}^{\times}$, i.e., an isomorphism (as we had switched to the
maximal inner groupoid of the category), the pushforward measure%
\[
Ha(G^{\prime})\longrightarrow Ha(G^{\prime})\qquad\qquad\mu\longmapsto
F_{\ast}\mu
\]
with $(F_{\ast}\mu)(X):=\mu(F^{-1}(X))$ is also a Haar measure: This is true
because (a), since the Borel $\sigma$-algebra is generated by open sets and
$F$ is continuous, $F^{-1}(X)$ of a measurable set $X\subseteq G$ is also
measurable in $G^{\prime}$, (b) $F^{-1}(X+g)=F^{-1}(X)+F^{-1}(g)$, so the
translation-invariance of $\mu$ implies that $F_{\ast}\mu$ is
translation-invariant, (c) the properties to be a finite measure on compact
sets, to be inner and outer regular, all just hinge on inclusion properties of
open or measurable sets, and since $F$ is a homeomorphism, these can all be
transported back and forth along $F$ and $F^{-1}$.

Since $F_{\ast}\mu$ is also a Haar measure, it pins down a unique element of
$Ha(G)$. This defines $Ha$ on morphisms.

(2) The integral in Eq. \ref{lg_2} requires justification: The function%
\[
x\mapsto\int_{G^{\prime}}f(x\xi)d\mu_{G^{\prime}}(\xi)
\]
is constant on each coset of the closed subgroup $G^{\prime}$ in $G$. Hence,
it defines a well-defined function on $G^{\prime\prime}$ by taking $x$ in Eq.
\ref{lg_2} to be any preimage of $\overline{x}\in G^{\prime\prime}$ in $x$.
Now use \cite[Theorem 2.49]{MR3444405}, using that the modular character
$\triangle$ (\cite[\S 2.4]{MR3444405}) is trivial for abelian groups as there
cannot be a difference between left and right Haar measures. A similar
discussion can be found in \cite[\S 33]{MR54173}. Axiom (4) of Def.
\ref{def_DeterminantFunctor} follows form a triple integral version of Eq.
\ref{lg_2},%
\begin{align*}
\int_{G_{3}}f(x)d\mu(x)  &  =\int_{G_{3}/G_{1}}\int_{G_{1}}f(x\xi)d\mu_{G_{1}%
}(\xi)d\mu_{G_{3}/G_{1}}(\overline{x})\text{ (by Eq. \ref{lg_2})}\\
&  =\int_{G_{3}/G_{2}}\int_{G_{2}/G_{1}}\int_{G_{1}}f(x\xi\zeta)d\mu_{G_{1}%
}(\xi)d\mu_{G_{2}/G_{1}}(\overline{\zeta})d\mu_{G_{3}/G_{2}}(\overline{x})
\end{align*}
by another use of Eq. \ref{lg_2}, where $x$ is a lift of $\overline{x}\in
G_{3}/G_{2}$ to $G_{3}$, and $\zeta$ is a lift of $\overline{\zeta}\in
G_{2}/G_{1}$ to $G_{2}$, where by $G_{2}\subseteq G_{3}$ it can be regarded an
element in $G_{3}$. Contracting the inner integrals along Eq. \ref{lg_2},%
\[
=\int_{G_{3}/G_{2}}\int_{G_{2}}f(x\zeta)d\mu_{G_{2}}(\zeta)d\mu_{G_{3}/G_{2}%
}(\overline{x})\text{,}%
\]
so that the left side corresponds to the upper left vertex in Diagram
\ref{l_CDetA} and the right sides of the three lines of this computation
correspond to the three remaining vertices.

Axiom (5) of Def. \ref{def_DeterminantFunctor} is harmless since the symmetry
constraint of $\mathsf{Tors}(\mathbb{R}_{>0}^{\times})$ is the identity.

\begin{example}
A distinctive feature of the Haar measure is its compatibility with exact
sequences for \emph{all} LCA groups, may they be discrete, compact, real or
$p$-adic, across all primes. This makes it well-defined even on the derived
category $\operatorname*{D}\nolimits_{\infty}^{b}(\mathsf{LCA})$
\cite{MR1914072}, \cite[\S 1.3, Corollary 2.1.1]{MR3302579}. For example,
multiplication by $5$ acts on%
\[
Ha(\mathbb{Q}_{5}^{A}\oplus\mathbb{R}^{B}\oplus\mathbb{Q}_{3}^{C}%
)\qquad\text{for any}\qquad A,B,C\in\mathbb{Z}_{\geq0}%
\]
by multiplication with $5^{B-A}$ since it stretches volumes in the reals,
shrinks them in the $5$-adics, and is volume-preserving on $\mathbb{Q}_{3}$.
Exactly whenever $A=B$, the map is volume-preserving as a whole. Writing the
same objects as cones%
\[
\mathbb{Q}_{5}^{A}\cong\operatorname*{cone}\left[  (\mathbb{Q}_{5}%
/\mathbb{Z}_{5})^{A}\longrightarrow\Sigma\mathbb{Z}_{5}^{A}\right]
\qquad\qquad\mathbb{R}^{B}\cong\operatorname*{cone}\left[  \mathbb{T}%
^{B}\longrightarrow\Sigma\mathbb{Z}^{B}\right]  \text{,}%
\]
multiplication by $5$ has a kernel of order $5^{A}$ on the $\mathbb{Q}%
_{5}/\mathbb{Z}_{5}$-summands, and a cokernel of order $5^{A}$ on the $\Sigma
$-shift of the $\mathbb{Z}_{5}$-summands (so that the total map has cone zero,
as it is necessary for an automorphism), resp. a kernel of order $5^{B}$ on
the $\mathbb{T}$-summands (where the map is a degree $5^{B}$ covering space)
and a cokernel of order $5^{B}$ on the $\Sigma$-shift of the $\mathbb{Z}%
$-summands. And on the $\operatorname*{cone}\left[  \mathbb{Q}_{3}%
/\mathbb{Z}_{3}\longrightarrow\Sigma\mathbb{Z}_{3}\right]  $ multiplication by
$5$ is an invertible map in both terms as $5\in\mathbb{Z}_{3}^{\times}$ is a unit.
\end{example}

\begin{example}
The graded determinant lines%
\begin{align*}
\det\colon\mathsf{Vect}_{fd}(F)  &  \longrightarrow\operatorname*{Pic}%
\nolimits_{F}^{\mathbb{Z}}\qquad\text{(}\mathsf{Vect}_{fd}(F)\text{ are
finite-dimensional }F\text{-vector spaces)}\\
V  &  \longmapsto\left(  \bigwedge\nolimits^{\dim V}V,\dim V\right)
\end{align*}
for a field $F$ can be regarded as being defined on full subcategories of
$\mathsf{LCA}$ in the cases where $F$ is a locally compact
field\footnote{Famously, this means that $F$ must be a finite extension of
$\mathbb{Q}_{p}$, $\mathbb{R}$, $\mathbb{F}_{p}((t))$ or a discrete field. All
finite-dimensional vector spaces over these fields are naturally objects in
$\mathsf{LCA}$.}. These can be richer than the Haar measure, e.g., the two
different pushforwards of any chosen trivialization of the $p$-adic
determinant line along the two arrows in Eq. \ref{l_h_1} would differ by
$\log_{p}(\ast)$, a possibly transcendental $p$-adic value. However, this
functor does not extend to all of $\mathsf{LCA}$. For example, there is no
sensible way to implement the exact sequence functoriality of Eq. \ref{lrio1}
to $\mathbb{Q}\hookrightarrow\mathbb{A}\twoheadrightarrow\mathbb{Q}^{\vee}$,
where $\mathbb{A}$ are the ad\`{e}les of $\mathbb{Q}$. It would require
intermingling real and $p$-adic determinant lines. The Haar measure
accomplishes this compatibility across all types over objects in
$\mathsf{LCA}$, but at the price of being less precise. For example, it is
oblivious to all the dimension gradings, and it \emph{has} to be since, for
example in $\mathbb{Q}\hookrightarrow\mathbb{A}\twoheadrightarrow
\mathbb{Q}^{\vee}$, $\mathbb{A}$ has a $1$-dimensional real number summand,
but neither $\mathbb{Q}$ nor $\mathbb{Q}^{\vee}$ has. See \cite[Prop.
13.3]{MR4028830} for details on what data is forgotten under switching to the
Haar measure.
\end{example}

\section{The rationalized Haar measure\label{sect_RationalizedHaarMeasure}}

In this section we describe the rationalized Haar measure as it occurs in
Theorem \ref{thm_RationalityOfHaarTorsor}. The basic question is:\ How can we
attach a Haar measure to a group such that it is well-defined up to a
\textit{rational} factor?\medskip

\textbf{Step 1:} Given an object $X\in\mathsf{LCA}_{\operatorname*{vf}}$, pick
a compact open subgroup $C\subseteq X$. This involves a choice, generally, but
is always possible by Prop. \ref{prop_struct}. This choice induces an exact
sequence%
\begin{equation}
\Sigma\colon C\hookrightarrow X\twoheadrightarrow X/C \label{lg_3}%
\end{equation}
in $\mathsf{LCA}$ with $X/C$ discrete (since $C$ was \textit{open} in $X$). On
the compact group, we may pick the canonical normalized Haar measure $\mu_{C}$
such that $\mu_{C}(C)=1$. This is possible because a Haar measure, by
definition, assigns a \textit{finite} volume to compact sets. On the discrete
group, we may pick the canonical counting measure, i.e., $\mu_{X/C}%
(\{\ast\})=1$ for any singleton set. This is tautologically a
translation-invariant measure. By Eq. \ref{lg_1} the exact sequence $\Sigma$
induces a natural isomorphism%
\[
Ha(\Sigma)\colon Ha(X)\overset{\sim}{\longrightarrow}Ha(C)\otimes Ha(X/C)
\]
and we define the\emph{ root measure} $\mu_{\operatorname*{root}}%
^{C}:=Ha(\Sigma)^{-1}(\mu_{C}\otimes\mu_{X/C})$. Equivalently, this the
\textit{unique} normalization of a Haar measure on $X$ such that Eq.
\ref{lg_2} holds, given that we use the measures $\mu_{C}$ and $\mu_{X/C}$ on
the compact and discrete piece.

(This might sound more complicated than it is: If we start with an arbitrary
Haar measure on $X$, we just need to rescale it to give $C$ the volume $+1$.
This yields exactly the measure just described).

\textbf{Step 2:} Now define%
\begin{align*}
Ha^{\mathbb{Q}}\colon\mathsf{LCA}_{\operatorname*{vf}}^{\times}  &
\longrightarrow\mathsf{Tors}(\mathbb{Q}_{>0}^{\times})\\
X  &  \longmapsto\mathbb{Q}_{>0}^{\times}\cdot\mu_{\operatorname*{root}}%
^{C}\text{,}%
\end{align*}
i.e., $Ha^{\mathbb{Q}}$ sends $X$ to the $\mathbb{Q}_{>0}^{\times}$-torsor of
all positive rational multiples of the root measure. We need to check that
this is well-defined: The root measure only depended on the choice of $C$ in
Eq. \ref{lg_3}. If we pick a further compact open $C^{\prime}$ such that
$C^{\prime}\subseteq C$, then $[C:C^{\prime}]:=\#(C/C^{\prime})$ is finite
(since $C$ is compact and $C^{\prime}$ open, so $C/C^{\prime}$ must be both
compact and discrete). We compute%
\begin{align*}
\int_{X}f(x)d\mu_{\operatorname*{root}}^{C}(x)  &  =\int_{X/C}\int_{C}%
f(x\xi)d\mu_{C}(\xi)d\mu_{X/C}(\overline{x})\\
&  =\sum_{\overline{x}\in X/C}\sum_{c\in C/C^{\prime}}\int_{C^{\prime}}f(x\xi
c)d\mu_{C}(\xi)\\
&  =\frac{1}{[C:C^{\prime}]}\sum_{\overline{x}\in X/C}\sum_{c\in C/C^{\prime}%
}\int_{C^{\prime}}f(x\xi c)d\mu_{C^{\prime}}(\xi)\\
&  =\frac{1}{[C:C^{\prime}]}\sum_{\overline{x}\in X/C^{\prime}}\int%
_{C^{\prime}}f(x\xi)d\mu_{C^{\prime}}(\xi)\\
&  =\frac{1}{[C:C^{\prime}]}\int_{X/C^{\prime}}\int_{C^{\prime}}f(x\xi
)d\mu_{C^{\prime}}(\xi)d\mu_{X/C^{\prime}}(\overline{x})\\
&  =\frac{1}{[C:C^{\prime}]}\int_{X}f(x)d\mu_{\operatorname*{root}}%
^{C^{\prime}}(x)
\end{align*}
for any function $f\in C_{c}(X,\mathbb{R})$, where the equalities are, in
succession, (1) Eq. \ref{lg_2} for $\mu_{\operatorname*{root}}^{C}$, (2)
$\mu_{X/C}$ is the counting measure, (3) $\mu_{C}=\frac{1}{[C:C^{\prime}]}%
\mu_{C^{\prime}}$ by our normalization that $\mu_{C}(C)=1$, and analogously
for $C^{\prime}$, (4) $C$ decomposes into $C^{\prime}$-tiles indexed over
$C/C^{\prime}$, (5) $\mu_{X/C^{\prime}}$ is the counting measure and (6) Eq.
\ref{lg_2} for $\mu_{\operatorname*{root}}^{C^{\prime}}$.

We learn that the two choices of the root measure only differ by a rational
factor, so both choices pin down the same $\mathbb{Q}_{>0}^{\times}$-subtorsor
$\mathbb{Q}_{>0}^{\times}\cdot\mu_{\operatorname*{root}}^{C}$ inside $Ha(X)$.
An entirely analogous argument works for compact objects $C^{\prime}$ such
that $C\subseteq C^{\prime}$ is bigger than the previous choice. All in all, a
zig-zag argument using that any two choices $C,C^{\prime}$ of compact opens in
$X$ have a joint compact sub-open, for example $C^{\prime}\cap C$, proves that
$Ha^{\mathbb{Q}}$ is well-defined on objects in $\mathsf{LCA}%
_{\operatorname*{vf}}$.\footnote{alternatively: instead of working with a
common subobject, it is also possible to work with $C+C^{\prime}$, a compact
open containing both constituents}

\begin{example}
The definition of $Ha^{\mathbb{Q}}$ cannot be extended to all of
$\mathsf{LCA}$: The group $\mathbb{R}$ does not possess a compact open
subgroup, so already Step 1 fails.
\end{example}

\begin{example}
$Ha^{\mathbb{Q}}\left(  \mathbb{F}_{q}((t))\right)  $ is the subset of all
Haar measures on the LCA\ group $\mathbb{F}_{q}((t))$ with the property that
$\operatorname*{vol}\left(  \mathbb{F}_{q}[[t]]\right)  $ is a positive
rational number.
\end{example}

It remains to show that this definition, so far only on objects, really
extends to a determinant functor satisfying all the axioms of Definition
\ref{def_DeterminantFunctor}.

Below, write $i\colon\mathbb{Q}_{>0}^{\times}\rightarrow\mathbb{R}%
_{>0}^{\times}$ for the inclusion of abelian groups and $i_{\ast}$ for the
induced basechange of Picard groupoids (Eq. \ref{lrio0}).

\begin{theorem}
\label{thm_HaQCharacterization}$Ha^{\mathbb{Q}}$ extends to a determinant
functor on $\mathsf{LCA}_{\operatorname*{vf}}$ such that the composition of
functors%
\[
\mathsf{LCA}_{\operatorname*{vf}}^{\times}\overset{Ha^{\mathbb{Q}%
}}{\longrightarrow}\mathsf{Tors}(\mathbb{Q}_{>0}^{\times})\overset{i_{\ast
}}{\longrightarrow}\mathsf{Tors}(\mathbb{R}_{>0}^{\times})
\]
agrees with the restriction of the Haar measure to the full subcategory
$\mathsf{LCA}_{\operatorname*{vf}}\subset\mathsf{LCA}$. As a result
$Ha^{\mathbb{Q}}$ can be described as follows: On every object $X$,
$Ha^{\mathbb{Q}}(X)$ singles out a subset of Haar measures in $Ha(X)$, and
under all arrows in $\mathsf{LCA}_{\operatorname*{vf}}$, the natural maps of
the Haar determinant functor $Ha$ respect these subsets.
\end{theorem}

This is a self-contained description of $Ha^{\mathbb{Q}}$ which saves us the
work to formulate the remaining axioms in Definition
\ref{def_DeterminantFunctor} for $Ha^{\mathbb{Q}}$:\ They all agree with the
ones for the usual Haar measure, just restricted to a subset of values. We
defer the proof to \S \ref{sect_ProofOfMainTheorem}.

\section{Computations}

\begin{lemma}
\label{lemma_p_1}The inclusion of finite abelian groups $\mathsf{Ab}%
_{\operatorname*{fin}}$ into all abelian groups $\mathsf{Ab}$ induces a
Verdier localization sequence%
\[
\operatorname*{D}\nolimits_{\infty}^{b}(\mathsf{Ab}_{\operatorname*{fin}%
})\longrightarrow\operatorname*{D}\nolimits_{\infty}^{b}(\mathsf{Ab}%
)\longrightarrow\operatorname*{D}\nolimits_{\infty}^{b}(\mathsf{Ab}%
/\mathsf{Ab}_{\operatorname*{fin}})\text{.}%
\]

\end{lemma}

\begin{proof}
It suffices to note that $\mathsf{Ab}_{\operatorname*{fin}}$ is a\ Serre
subcategory of the abelian category $\mathsf{Ab}$.
\end{proof}

\begin{lemma}
\label{lemma_p_2}The inclusion of compact abelian groups $\mathsf{C}$ into
$\mathsf{LCA}_{\operatorname*{vf}}$ induces a Verdier localization sequence,
up to equivalence of the shape,%
\[
\operatorname*{D}\nolimits_{\infty}^{b}(\mathsf{C})\longrightarrow
\operatorname*{D}\nolimits_{\infty}^{b}(\mathsf{LCA}_{\operatorname*{vf}%
})\longrightarrow\operatorname*{D}\nolimits_{\infty}^{b}(\mathsf{Ab}%
/\mathsf{Ab}_{\operatorname*{fin}})\text{.}%
\]
This equivalence is given by an exact equivalence of exact categories
$\Xi\colon\mathsf{Ab}/\mathsf{Ab}_{\operatorname*{fin}}\overset{\sim
}{\longrightarrow}\mathsf{LCA}_{\operatorname*{vf}}/\mathsf{C}$.
\end{lemma}

\begin{proof}
Again, we note that $\mathsf{C}$ is a Serre subcategory of $\mathsf{LCA}%
_{\operatorname*{vf}}$. Since $\mathsf{LCA}_{\operatorname*{vf}}$ is not an
abelian category, one needs to verify a few more properties to obtain the
Verdier localization sequence%
\[
\operatorname*{D}\nolimits_{\infty}^{b}(\mathsf{C})\longrightarrow
\operatorname*{D}\nolimits_{\infty}^{b}(\mathsf{LCA}_{\operatorname*{vf}%
})\longrightarrow\operatorname*{D}\nolimits_{\infty}^{b}(\mathsf{LCA}%
_{\operatorname*{vf}}/\mathsf{C})\text{,}%
\]
and one may follow \cite[Corollary 5.11]{hr2}. Hence, our claim is proven once
we exhibit an exact equivalence of exact categories%
\begin{equation}
\Xi\colon\mathsf{Ab}/\mathsf{Ab}_{\operatorname*{fin}}\longrightarrow
\mathsf{LCA}_{\operatorname*{vf}}/\mathsf{C}\text{.} \label{lp_3}%
\end{equation}
To this end, we need to describe the quotient categories on either side. In
the setting of \cite{hr2}, $\mathsf{C}$ is inflation-percolating in
$\mathsf{LCA}_{\operatorname*{vf}}$, so by \cite[Prop. 4.4]{hr2} the system
$S_{\mathsf{C}}$ generated by (1) admissible epics with kernel in $\mathsf{C}%
$, (2) admissible monics with cokernel in $\mathsf{C}$, (3) any finite
composition thereof, is \textit{left} multiplicative\footnote{Unfortunately,
some authors use left and right with opposite meaning in the context of a
calculus of fractions. We follow the convention of Kashiwara--Shapira
\cite[Remark 7.1.8]{MR2182076}.}, so $\mathsf{LCA}_{\operatorname*{vf}%
}/\mathsf{C}$ can be modelled through right roofs \cite[Remark 7.1.7]%
{MR2182076}. We note that for Serre subcategories in abelian categories,
$\mathsf{Ab}/\mathsf{Ab}_{\operatorname*{fin}}=\mathsf{Ab}[S_{\mathsf{Ab}%
_{\operatorname*{fin}}}^{-1}]$ is a localization by a both left and right
multiplicative system. The functor $\Xi$ sends an abelian group to itself,
equipped with the discrete topology. \textit{(Fullness and essential
surjectivity)} Consider an arbitrary right roof in $\mathsf{LCA}%
_{\operatorname*{vf}}/\mathsf{C}=\mathsf{LCA}_{\operatorname*{vf}%
}[S_{\mathsf{C}}^{-1}]$ between discrete abelian groups $X,Y$. It has the
shape of the solid arrows in
\begin{equation}%
{
\begin{tikzcd}
	&& D \\
	\\
	X && D && Y \\
	\\
	&& Z
	\arrow["{f'}", dashed, from=3-1, to=1-3]
	\arrow[dashed, from=3-1, to=3-3]
	\arrow["f"', from=3-1, to=5-3]
	\arrow[equals, dashed, from=3-3, to=1-3]
	\arrow["{s'}"', dashed, from=3-5, to=1-3]
	\arrow["{s'}", dashed, from=3-5, to=3-3]
	\arrow["{s\in S_{\mathsf{C} }}", from=3-5, to=5-3]
	\arrow["q", dashed, two heads, from=5-3, to=3-3]
\end{tikzcd}
}
\label{diag_p_1}%
\end{equation}
(ignore the dashed arrows for the moment). Since $Z\in\mathsf{LCA}%
_{\operatorname*{vf}}$, there must exist a compact clopen $C\subseteq Z$ and
we obtain an exact sequence%
\begin{equation}
C\hookrightarrow Z\overset{q}{\twoheadrightarrow}D \label{l_0d}%
\end{equation}
with $D$ discrete in $\mathsf{LCA}_{\operatorname*{vf}}$ (Prop.
\ref{prop_struct}). Now we may add all the dashed arrows in Diagram
\ref{diag_p_1}: $f^{\prime}:=q\circ f$ and $s^{\prime}:=q\circ s$. Since $s\in
S_{\mathsf{C}}$ and $q$ is an admissible epic with kernel in $\mathsf{C}$, it
follows that $s^{\prime}\in S_{\mathsf{C}}$. As a result, Diagram
\ref{diag_p_1} determines a valid equivalence of right roofs and shows that we
may assume that $Z$ was discrete to start with. It follows that $s\in
S_{\mathsf{C}}$ is a morphism between two discrete groups. However, then we
must have that $s\in S_{\mathsf{Ab}_{\operatorname*{fin}}}$ since finite
groups are the only ones which are simultaneously discrete and compact. It
follows that the right roof lies is the image under the functor $\Xi$ of a
right roof in $\mathsf{Ab}/\mathsf{Ab}_{\operatorname*{fin}}$. Thus, $\Xi$ is
a full functor. For essential surjectivity, note that if $Z\in\mathsf{LCA}%
_{\operatorname*{vf}}$ is an arbitrary object, the same exact sequence as in
Eq. \ref{l_0d} shows that $q\colon Z\simeq\Xi(D)$ is an isomorphism in the
quotient category $\mathsf{LCA}_{\operatorname*{vf}}/\mathsf{C}$.
\textit{(Faithfulness)} Suppose a right roof in $\mathsf{Ab}[S_{\mathsf{Ab}%
_{\operatorname*{fin}}}^{-1}]$ that is equivalent to the zero map in
$\mathsf{Ab}[S_{\mathsf{Ab}_{\operatorname*{fin}}}^{-1}]$. Then since
$S_{\mathsf{Ab}_{\operatorname*{fin}}}\subseteq S_{\mathsf{C}}$, the same
equivalence of roof also shows that the roof is equivalent to the zero map in
$\mathsf{LCA}_{\operatorname*{vf}}/\mathsf{C}$. \textit{(Exactness and
reflection of exactness)} The functor is induced from the functor
$\mathsf{Ab}\longrightarrow\mathsf{LCA}_{\operatorname*{vf}}$, which is
obviously exact, to the quotient category $\mathsf{Ab}/\mathsf{Ab}%
_{\operatorname*{fin}}$ by the universal property of sending $\mathsf{Ab}%
_{\operatorname*{fin}}$ to zero objects. Restricted onto the strict image, it
is clear that $\Xi$ reflects exactness since exactness on discrete groups in
$\mathsf{LCA}_{\operatorname*{vf}}$ reduces to exactness of the underlying
abelian groups. There is no topology to take into consideration.
\end{proof}

\begin{lemma}
\label{lemma_p_3}There is an equivalence of localizing non-commutative motives%
\[
\mathcal{U}^{\operatorname*{loc}}(\mathsf{LCA}_{\operatorname*{vf}%
})\overset{\sim}{\longrightarrow}\Sigma\mathcal{U}^{\operatorname*{loc}%
}(\mathsf{Ab}_{\operatorname*{fin}})\text{.}%
\]

\end{lemma}

There is no urgent need for non-commutative motives, this is only one possible
way to lead us to Corollary \ref{cor_1}, which is all that we shall truly need.

\begin{proof}
From Lemma \ref{lemma_p_1} and Lemma \ref{lemma_p_2} we get the fiber
sequences of localizing non-commutative motives%
\begin{align*}
\mathcal{U}^{\operatorname*{loc}}(\mathsf{Ab}_{\operatorname*{fin}})  &
\longrightarrow\mathcal{U}^{\operatorname*{loc}}(\mathsf{Ab})\longrightarrow
\mathcal{U}^{\operatorname*{loc}}(\mathsf{Ab}/\mathsf{Ab}_{\operatorname*{fin}%
})\\
\mathcal{U}^{\operatorname*{loc}}(\mathsf{C})  &  \longrightarrow
\mathcal{U}^{\operatorname*{loc}}(\mathsf{LCA}_{\operatorname*{vf}%
})\longrightarrow\mathcal{U}^{\operatorname*{loc}}(\mathsf{Ab}/\mathsf{Ab}%
_{\operatorname*{fin}})\text{.}%
\end{align*}
By Tychonov's theorem arbitrary products of compact abelian groups are again
compact, so $\mathsf{C}$ is a complete category. Hence $\mathcal{U}%
^{\operatorname*{loc}}(\mathsf{C})=0$ by the Eilenberg swindle. Analogously,
$\mathsf{Ab}$ is co-complete, so $\mathcal{U}^{\operatorname*{loc}%
}(\mathsf{Ab})=0$. Combining these facts, the fiber sequences simplify to%
\[
\mathcal{U}^{\operatorname*{loc}}(\mathsf{LCA}_{\operatorname*{vf}%
})\overset{\sim}{\longrightarrow}\mathcal{U}^{\operatorname*{loc}}%
(\mathsf{Ab}/\mathsf{Ab}_{\operatorname*{fin}})\overset{\sim}{\longrightarrow
}\Sigma\mathcal{U}^{\operatorname*{loc}}(\mathsf{Ab}_{\operatorname*{fin}%
})\text{.}%
\]

\end{proof}

\begin{remark}
Actually, $\mathsf{Ab}$ and $\mathsf{C}$ are both complete and co-complete, so
it makes no difference which property we use for the Eilenberg swindle.
However, in the format as described in the proof, the inclusion functor
$\mathsf{Ab}\rightarrow\mathsf{LCA}_{\operatorname*{vf}}$ preserves arbitrary
colimits (resp. $\mathsf{C}\rightarrow\mathsf{LCA}_{\operatorname*{vf}}$
preserves arbitrary limits), so it is easier to visualize what is happening.
The respectively opposite type of (co)limit is not preserved by these functors.
\end{remark}

As a side result, we have computed the entire non-connective $K$-theory
spectrum of $\mathsf{LCA}_{\operatorname*{vf}}$:

\begin{corollary}
\label{cor_1}$K(\mathsf{LCA}_{\operatorname*{vf}})\cong\Sigma K(\mathsf{Ab}%
_{\operatorname*{fin}})$.
\end{corollary}

This differs significantly from the counterpart where real vector spaces are allowed:

\begin{theorem}
[Clausen]\label{thm_clausen}$K(\mathsf{LCA})\cong\operatorname*{cofib}%
(K(\mathbb{Z})\rightarrow K(\mathbb{R}))$.
\end{theorem}

Clausen's original proof is in \cite{clausen}. Another proof is in
\cite{MR4028830}.

\section{Proof of the main theorem\label{sect_ProofOfMainTheorem}}

\begin{theorem}
\label{thm1}We have%
\[
K_{1}(\mathsf{LCA}_{\operatorname*{vf}})\cong\mathbb{Q}_{>0}^{\times}%
\qquad\text{and}\qquad K_{1}(\mathsf{LCA})\cong\mathbb{R}_{>0}^{\times}%
\]
and under the exact functor $\mathsf{LCA}_{\operatorname*{vf}}\longrightarrow
\mathsf{LCA}$, the induced map on $K_{1}$ is the inclusion of rational
numbers,%
\[
\mathbb{Q}_{>0}^{\times}\subset\mathbb{R}_{>0}^{\times}\text{.}%
\]
Moreover,%
\[
K_{0}(\mathsf{LCA}_{\operatorname*{vf}})=K_{0}(\mathsf{LCA})=0\text{.}%
\]

\end{theorem}

\begin{proof}
The computation $K_{1}(\mathsf{LCA})\cong\mathbb{R}_{>0}^{\times}$ goes back
to Clausen \cite{clausen} (use Theorem \ref{thm_clausen}, showing that
\ldots$\rightarrow\mathbb{Z}^{\times}\rightarrow\mathbb{R}^{\times}\rightarrow
K_{1}(\mathsf{LCA})\overset{0}{\rightarrow}\mathbb{Z}$ is exact). A different
proof is in \cite[Theorem 12.8]{MR4028830}. The same techniques show that
$K_{0}(\mathsf{LCA})=0$. Hence, we focus on computing the vector-free variant
$K_{1}(\mathsf{LCA}_{\operatorname*{vf}})$: Every finite abelian group
uniquely(!) splits into its $p$-primary torsion summands. This induces an
equivalence of abelian categories,%
\[
\mathsf{Ab}_{\operatorname*{fin}}\cong\bigoplus_{p}\mathsf{Ab}%
_{\operatorname*{fin}}[p^{\infty}]\text{,}%
\]
where $\mathsf{Ab}_{\operatorname*{fin}}[p^{\infty}]$ is the abelian category
of finite $p$-power torsion abelian groups. Every such group has a finite
filtration by quotients killed by $p$ (or said differently: the simple objects
of the category). Hence, these quotients are finite-dimensional $\mathbb{F}%
_{p}$-vector spaces. By d\'{e}vissage we deduce for connective (Quillen)
$K$-theory that%
\[
K^{\operatorname*{conn}}(\mathsf{Ab}_{\operatorname*{fin}}[p^{\infty}])\cong
K^{\operatorname*{conn}}(\mathbb{F}_{p})
\]
for all primes $p$ (\cite[Ch. V, Theorem 4.1]{MR3076731}). Since both
$\mathsf{Ab}_{\operatorname*{fin}}[p^{\infty}]$ and $\mathsf{Vect}%
_{fd}(\mathbb{F}_{p})$ are Noetherian abelian categories, Schlichting's
theorem \cite[\S 10.1, Theorem 7]{MR2206639} implies that either category has
connective non-connective $K$-theory (i.e., $\pi_{i}K(-)=0$ for all $i<0$).
Since the categories are abelian, they are idempotent complete, so
$K_{0}=K_{0}^{\operatorname*{conn}}$ (\cite[\S 6.2, Remark 3]{MR2206639}). It
follows that for either category connective $K$-theory agrees with the
non-connective $K$-theory \cite[\S 12.2]{MR2206639}, so%
\begin{equation}
K(\mathsf{Ab}_{\operatorname*{fin}})\cong\bigoplus_{p}K(\mathsf{Ab}%
_{\operatorname*{fin}}[p^{\infty}])\cong\bigoplus_{p}K(\mathbb{F}_{p}%
)\text{.}\label{lp_1}%
\end{equation}
Combining this with Corollary \ref{cor_1},
\begin{equation}
K_{1}(\mathsf{LCA}_{\operatorname*{vf}})\underset{\text{Cor. \ref{cor_1}%
}}{\cong}K_{0}(\mathsf{Ab}_{\operatorname*{fin}})\underset{\text{Eq.
\ref{lp_1}}}{\cong}\bigoplus_{p}K_{0}(\mathbb{F}_{p})\cong\bigoplus
_{p}\mathbb{Z}\text{.}\label{lp_2}%
\end{equation}
The trickier part of the proof is now that while $\mathsf{LCA}%
_{\operatorname*{vf}}\longrightarrow\mathsf{LCA}$ certainly induces a map%
\begin{equation}%
{
\begin{tikzcd}
	{K_{1}(\mathsf{LCA}_{\operatorname{vf} })} && {K_{1}(\mathsf{LCA})} \\
	\\
	{\underset{p}{\bigoplus} \mathbb{Z}} && {\mathbb{R}^{\times}_{>0},}
	\arrow[from=1-1, to=1-3]
	\arrow[from=1-1, to=3-1]
	\arrow[from=1-3, to=3-3]
	\arrow[dashed, from=3-1, to=3-3]
\end{tikzcd}
}%
\label{d_3}%
\end{equation}
it is not so clear what the dashed arrow actually does (for all we know so
far, it could be the zero map). In order to analyze this map, we shall
construct a class $-\alpha_{p}\in K_{1}(\mathsf{LCA}_{\operatorname*{vf}})$
which is sent under the map in Eq. \ref{lp_2} to $1_{p}$ (i.e., $1$ in the
summand belonging to the prime $p$). To this end, we follow a technique due to
Sherman and we rephrase the map of Corollary \ref{cor_1} in terms of a
$K$-theory localization sequence, where it corresponds to the connecting map
$\delta$ in%
\begin{equation}
K^{\operatorname*{conn}}(\mathsf{Ab}_{\operatorname*{fin}})\longrightarrow
\underset{=0}{K^{\operatorname*{conn}}(\mathsf{Ab})}\longrightarrow
K^{\operatorname*{conn}}(\mathsf{Ab}/\mathsf{Ab}_{\operatorname*{fin}%
})\overset{\delta}{\longrightarrow}\Sigma K^{\operatorname*{conn}}%
(\mathsf{Ab}_{\operatorname*{fin}})\label{lc4}%
\end{equation}
in connective $K$-theory (this is only true because we work on $K_{1}$,
mapping to $K_{0}$, and we had seen above that in this range connective
$K$-theory agrees with its non-connective counterpart). We now claim that the
automorphism of multiplication by $p$ on the $p$-adics $\mathbb{Q}_{p}$ is a
representative for the sought-for class $\alpha_{p}$. To see this, we need to
compute $\tilde{\delta}(\alpha_{p})$ in the following commutative diagram:%
\[%
{
\begin{tikzcd}
	{\alpha_p \in} & {\pi_{1}K(\mathsf{LCA}_{\operatorname*{vf}})} \\
	& {\pi_{1}K^{\operatorname*{conn}}(\mathsf{LCA}_{\operatorname*{vf}}%
/\mathsf{C})} \\
	& {\pi_{1}K^{\operatorname*{conn}}(\mathsf{Ab}/\mathsf{Ab}_{\operatorname
*{fin} })} && {\pi_{1}\Sigma K^{\operatorname*{conn}}(\mathsf{Ab}%
_{\operatorname*{fin} })} & {\cong\bigoplus_{p}\mathbb{Z}\ni-1_{p}.}
	\arrow[from=1-2, to=2-2]
	\arrow["{\tilde{\delta}}", from=1-2, to=3-4]
	\arrow["{{\Xi}^{-1}}"', from=2-2, to=3-2]
	\arrow["\delta"', from=3-2, to=3-4]
\end{tikzcd}
}%
\]
The object $\mathbb{Q}_{p}\in\mathsf{LCA}_{\operatorname*{vf}}$ under
$\Xi^{-1}$ (Eq. \ref{lp_3}) corresponds to the discrete group $\mathbb{Q}%
_{p}/\mathbb{Z}_{p}$ in $\mathsf{Ab}$ since%
\[
\mathbb{Z}_{p}\hookrightarrow\mathbb{Q}_{p}\overset{w_{p}}{\twoheadrightarrow
}\mathbb{Q}_{p}/\mathbb{Z}_{p}%
\]
is an exact sequence in $\mathsf{LCA}_{\operatorname*{vf}}$ with
$\mathbb{Z}_{p}$ compact, so that $w_{p}$ becomes an isomorphism in the
quotient category $\mathsf{LCA}_{\operatorname*{vf}}/\mathsf{C}$. As
$\mathbb{Q}_{p}/\mathbb{Z}_{p}$ is already a discrete group, $\Xi^{-1}$ just
sends this group to itself, now living in the quotient category $\mathsf{Ab}%
/\mathsf{Ab}_{\operatorname*{fin}}$. The horizontal map $\delta$ comes from
the connecting homomorphism in the long exact sequence of homotopy groups
attached to the fibration in Eq. \ref{lc4}. We can compute this explicitly by
using a simplicial model of the underlying $K$-theory spaces\footnote{note
that since we use \textit{connective} $K$-theory here, we may regard each
$K^{\operatorname*{conn}}$ as a space, equipped with the datum of an infinite
loop space/grouplike $E_{\infty}$-space}. We choose the Gillet--Grayson model
$G_{\bullet}$ for this purpose and in \S \ref{sect_GilletGraysonModel} we
summarize all the facts we shall need to know in order to carry out the
following computation. The map of multiplication by $p$ on $\mathbb{Q}%
_{p}/\mathbb{Z}_{p}$ corresponds to the arrow%
\begin{equation}
\mathbb{Q}_{p}/\mathbb{Z}_{p}\overset{\cdot p}{\longrightarrow}\mathbb{Q}%
_{p}/\mathbb{Z}_{p}\label{lcz62}%
\end{equation}
which is indeed an automorphism in $\mathsf{Ab}/\mathsf{Ab}%
_{\operatorname*{fin}}$ (it is surjective, but has kernel $\frac{1}%
{p}\mathbb{Z}_{p}/\mathbb{Z}_{p}$ in $\mathsf{Ab}$. Since this kernel is a
finite abelian group, it is zero in the quotient category). This determines a
class $\alpha_{p}\in\pi_{1}K^{\operatorname*{conn}}(\mathsf{Ab}/\mathsf{Ab}%
_{\operatorname*{fin}})$, geometrically representable by a closed loop around
the basepoint $(0,0)$ in the Gillet--Grayson model. The boundary map $\delta$
on homotopy groups%
\[
\pi_{1}\left\vert G_{\bullet}(\mathsf{Ab}/\mathsf{Ab}_{\operatorname*{fin}%
})\right\vert \longrightarrow\pi_{0}\left\vert G_{\bullet}(\mathsf{Ab}%
_{\operatorname*{fin}})\right\vert
\]
corresponds to lifting this loop to a path in $G_{\bullet}(\mathsf{Ab})$ and
the output value of $\delta$ is the connected component in which the lifted
path ends in $\pi_{0}\left\vert G_{\bullet}(\mathsf{Ab}_{\operatorname*{fin}%
})\right\vert $. The homotopy lifting property of a fibration usually
guarantees that such a lift exists for any loop. Unfortunately, even though
Eq. \ref{lc4} is a (homotopy) fibration sequence, the underlying map between
the Gillet--Grayson simplicial sets need not be a simplicial (Kan) fibration.
However, we may entirely bypass the path lifting property if we are able to
manually exhibit a path lifting the loop.{\footnote{I thank Clayton Sherman
and Alexander Nenashev, as I have learned this technique of working with
explicit simplicial models of $K$-theory from their highly inspiring works
\cite{MR1409623,MR1409625,MR1621689}.}} This is what we shall do: A
$1$-simplex in $G_{\bullet}(\mathsf{Ab})$ is given by a pair of exact
sequences in $\mathsf{Ab}$ with the same cokernel (Eq. \ref{lcimez22}), so
define $\xi$ by%
\[%
{
\begin{tikzcd}
	{\frac{1}{p}\mathbb{Z}_{p}/\mathbb{Z}_{p}} && {\mathbb{Q}_{p}/\mathbb{Z}_{p}}
&& {\mathbb{Q}_{p}/\mathbb{Z}_{p}} \\
	0 && {\mathbb{Q}_{p}/\mathbb{Z}_{p}} && {\mathbb{Q}_{p}/\mathbb{Z}_{p},}
	\arrow[dotted, hook, from=1-1, to=1-3]
	\arrow["{\cdot p}", dotted, two heads, from=1-3, to=1-5]
	\arrow[equals, from=1-5, to=2-5]
	\arrow[hook, from=2-1, to=2-3]
	\arrow["{\cdot1}"', two heads, from=2-3, to=2-5]
\end{tikzcd}
}%
\]
a path from $(\frac{1}{p}\mathbb{Z}_{p}/\mathbb{Z}_{p},0)$ to $(\mathbb{Q}%
_{p}/\mathbb{Z}_{p},\mathbb{Q}_{p}/\mathbb{Z}_{p})$. Now consider the path
given solely by the \textit{solid} arrows in%
\begin{equation}%
{
\begin{tikzcd}
	{(\frac{1}{p}\mathbb{Z}_{p}/\mathbb{Z}_{p},0)} &&&& {(\mathbb{Q}_{p}%
/\mathbb{Z}_{p},\mathbb{Q}_{p}/\mathbb{Z}_{p})} \\
	&& {(0,0)}
	\arrow["\xi", from=1-1, to=1-5]
	\arrow["\rho", dashed, from=2-3, to=1-1]
	\arrow["{\nu(\mathbb{Q}_{p}/\mathbb{Z}_{p})}"', from=2-3, to=1-5]
\end{tikzcd}
}%
\label{lcz61}%
\end{equation}
and $\nu$ as in Eq. \ref{lcz60}. We shall show that under the map $\left\vert
G_{\bullet}(\mathsf{Ab})\right\vert \longrightarrow\left\vert G_{\bullet
}(\mathsf{Ab}/\mathsf{Ab}_{\operatorname*{fin}})\right\vert $ this path
(essentially) maps to a closed loop: We note that both (distinct) objects $0$
and $\frac{1}{p}\mathbb{Z}_{p}/\mathbb{Z}_{p}$ from $\mathsf{Ab}$ become zero
objects in the quotient category $\mathsf{Ab}/\mathsf{Ab}_{\operatorname*{fin}%
}$, and thus become isomorphic by a unique map. Therefore, up to replacing
$\mathsf{Ab}/\mathsf{Ab}_{\operatorname*{fin}}$ by an equivalent category,
call it $\widetilde{\mathsf{Ab}/\mathsf{Ab}_{\operatorname*{fin}}}$, we may
identify these zero objects to become strictly the same object. Then
$(\frac{1}{p}\mathbb{Z}_{p}/\mathbb{Z}_{p},0)$ and $(0,0)$ are the same
$0$-simplex in $G_{\bullet}(\widetilde{\mathsf{Ab}/\mathsf{Ab}%
_{\operatorname*{fin}}})$. From the discussion in Example \ref{example_pi1GG}
we now see that all arrows in Diagram \ref{lcz61} define a closed loop around
the basepoint $(0,0)$ of $\left\vert G_{\bullet}(\widetilde{\mathsf{Ab}%
/\mathsf{Ab}_{\operatorname*{fin}}})\right\vert $ representing the $K_{1}%
$-class of the automorphism of Eq. \ref{lcz62}. Hence, the solid arrows in
Diagram \ref{lcz61} yield a lift of this path to $\left\vert G_{\bullet
}(\mathsf{Ab})\right\vert $.%
\[%
{\includegraphics[
height=1.4218in,
width=2.9049in
]%
{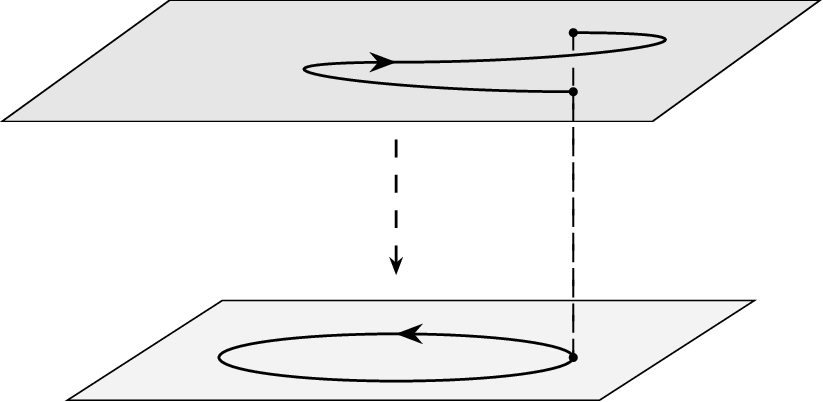}%
}
\]
The endpoint of this path in $G_{\bullet}(\mathsf{Ab}_{\operatorname*{fin}})$
is the $0$-simplex $(\frac{1}{p}\mathbb{Z}_{p}/\mathbb{Z}_{p},0)$, which
corresponds to the $K_{0}$-class%
\[
\left[  0\right]  -\left[  \frac{1}{p}\mathbb{Z}_{p}/\mathbb{Z}_{p}\right]
\]
by Example \ref{example_pi0GG}. Under the d\'{e}vissage of Eq. \ref{lp_1},
this in turn identifies\footnote{d\'{e}vissage is trivial on this class as the
underlying object is already simple} with (the $K_{0}$-group negative of) a
one-dimensional $\mathbb{F}_{p}$-vector space $\mathbb{F}_{p}\simeq\frac{1}%
{p}\mathbb{Z}_{p}/\mathbb{Z}_{p}$, i.e., to $-1\in\mathbb{Z}$ in the $p$-th
direct summand of Eq. \ref{lp_2}, all on the right. This finishes the proof
that%
\[
\delta([\alpha_{p}])=-1_{p}\in\bigoplus_{p}\mathbb{Z}\text{.}%
\]
We now need to check what $\alpha_{p}$ corresponds to under the map%
\[
K_{1}(\mathsf{LCA}_{\operatorname*{vf}})\longrightarrow K_{1}(\mathsf{LCA})
\]
induced from the exact functor $\mathsf{LCA}_{\operatorname*{vf}%
}\longrightarrow\mathsf{LCA}$. However, it was already computed in
\cite[Example 2.3, Prop. 13.3]{MR4028830} that multiplication by $p$ on
$\mathbb{Q}_{p}$ corresponds under the Haar torsor to multiplication with the
$p$-adic valuation, i.e., $v_{p}(p)=\frac{1}{p}\in\mathbb{R}_{>0}^{\times}$.
This means that the map%
\begin{align*}
K_{1}(\mathsf{LCA}_{\operatorname*{vf}}) &  \longrightarrow K_{1}%
(\mathsf{LCA})\\
\bigoplus_{p}\mathbb{Z} &  \longrightarrow\mathbb{R}_{>0}^{\times}%
\end{align*}
agrees with $-1_{p}\mapsto\frac{1}{p}$. But then it is actually better to
identify $\bigoplus_{p}\mathbb{Z}\simeq\mathbb{Q}_{>0}^{\times}$ with the
positive rational numbers whose prime factor decomposition $+2^{a_{2}}%
3^{a_{3}}\ldots$ corresponds to the vector $(a_{2},a_{3},\ldots)$ in
$\bigoplus_{p}\mathbb{Z}$.
\end{proof}

Write $\mathsf{Sp}^{0,1}$ for the stable $\infty$-category of spectra
concentrated in degrees $[0,1]$. Recall that we denote by $\mathsf{Picard}$
the $2$-category of Picard groupoids.

\begin{proposition}
\label{prop_HomotopyCatsOfPicardGroupoidsAndSpectra}There is an equivalence of
homotopy categories%
\begin{equation}
\Psi\colon\operatorname*{Ho}(\mathsf{Picard})\overset{\sim}{\longrightarrow
}\operatorname*{Ho}(\mathsf{Sp}^{0,1})\text{.} \label{lceyAJ1}%
\end{equation}
This correspondence preserves the notions of homotopy groups $\pi_{0},\pi_{1}$
on either side and the stable $k$-invariant of the spectrum corresponds to the
stable $k$-invariant for Picard groupoids of Definition \ref{def_Signature} in
$\mathsf{Picard}$.
\end{proposition}

To elaborate on the stable $k$-invariant: Given a spectrum concentrated in
degrees $[0,1]$, let%
\[
\underset{\text{or }K(\pi_{1}X,1)}{\Sigma H\pi_{1}(X)}\longrightarrow
X\longrightarrow\underset{\text{or }K(\pi_{0}X,0)}{H\pi_{0}(X)}\text{,}%
\]
be its fiber sequence of truncation in $\mathsf{Sp}$, decomposing $X$ into two
(shifts) of Eilenberg--Mac\ Lane spectra (a tiny version of a stable Postnikov
tower). Then the connecting homomorphism%
\[
H\pi_{0}(X)\longrightarrow\Sigma^{2}H\pi_{1}(X)
\]
determines a class in $[H\pi_{0}(X),\Sigma^{2}H\pi_{1}(X)]$, which as an
abelian group can be seen to correspond to the group of homomorphism $\pi
_{0}(X)\otimes\mathbb{Z}/2\rightarrow\left.  _{2}\pi_{1}(X)\right.  $, as the
stable $k$-invariant of the attached Picard groupoid. Proofs are given in
\cite[\S 5.1, Theorem 5.3]{MR2981817} or \cite[1.5\ Theorem]{MR2981952}, but
already\ Grothendieck was aware of this correspondence.

\begin{example}
\label{example_TorsAIsLikeSigmaHA}$\Psi\left(  \mathsf{Tors}(A)\right)
=\Sigma HA$ for all abelian groups. Said differently: The groupoid of
$A$-torsors corresponds to the Eilenberg-Mac Lane spectrum of $A$, shifted to
sit in degree one.
\end{example}

\begin{example}
\label{example_BaseChangeOfTorsAgreesWithNaturalMap}The basechange of torsors
from Eq. \ref{lrio0}, $i_{\ast}\colon\mathsf{Tors}(A)\rightarrow
\mathsf{Tors}(B)$, under $\Psi$ gets sent to the $\Sigma$-shift of the natural
map $HA\rightarrow HB$.
\end{example}

\begin{proposition}
\label{prop_1}The virtual objects $V(\mathsf{LCA}_{\operatorname*{vf}})$ are
symmetric monoidally equivalent to the Picard groupoid $\mathsf{Tors}%
(\mathbb{Q}_{>0}^{\times})$. Under this identification, the composition%
\[
\mathsf{LCA}_{\operatorname*{vf}}^{\times}\overset{u}{\longrightarrow
}\mathsf{Tors}(\mathbb{Q}_{>0}^{\times})\overset{i_{\ast}}{\longrightarrow
}\mathsf{Tors}(\mathbb{R}_{>0}^{\times})\text{,}%
\]
where $u$ is the universal determinant functor of $\mathsf{LCA}%
_{\operatorname*{vf}}$, agrees with the Haar measure restricted to
$\mathsf{LCA}_{\operatorname*{vf}}$, i.e.,%
\begin{equation}
i_{\ast}\circ u=Ha\mid_{\mathsf{LCA}_{\operatorname*{vf}}}\text{.} \label{lt3}%
\end{equation}

\end{proposition}

\begin{proof}
Following Deligne, the universal determinant functor of an exact category
$\mathsf{C}$ can be modelled through the virtual objects $V(\mathsf{C})$ of
\cite{MR902592}. This is the Picard groupoid belonging to truncated connective
$K$-theory under the correspondence of homotopy categories of Prop.
\ref{prop_HomotopyCatsOfPicardGroupoidsAndSpectra}, i.e.,%
\begin{align*}
\pi_{0}V(\mathsf{LCA}_{\operatorname*{vf}})  &  \cong\pi_{0}%
K^{\operatorname*{conn}}(\mathsf{LCA}_{\operatorname*{vf}})=0\text{,}\\
\pi_{1}V(\mathsf{LCA}_{\operatorname*{vf}})  &  \cong\pi_{1}%
K^{\operatorname*{conn}}(\mathsf{LCA}_{\operatorname*{vf}})\cong%
\mathbb{Q}_{>0}^{\times}\text{,}%
\end{align*}
both by Theorem \ref{thm1}. See also \cite{MR3302579} for more background on
the link between $K$-theory and $V(\mathsf{C})$. A Picard groupoid (as well as
a spectrum concentrated in degree $[0,1]$) is uniquely determined by these
values and the stable $k$-invariant%
\[
\pi_{0}K^{\operatorname*{conn}}(\mathsf{LCA}_{\operatorname*{vf}}%
)\otimes\mathbb{Z}/2\longrightarrow\pi_{1}K^{\operatorname*{conn}%
}(\mathsf{LCA}_{\operatorname*{vf}})\text{.}%
\]
Since $\pi_{0}K^{\operatorname*{conn}}(\mathsf{LCA}_{\operatorname*{vf}})=0$,
this map is necessarily zero. Hence, it follows that $V(\mathsf{LCA}%
_{\operatorname*{vf}})$ has trivial symmetry constraint. Just by comparison of
invariants, we deduce that $V(\mathsf{LCA}_{\operatorname*{vf}})\cong%
\mathsf{Tors}(\mathbb{Q}_{>0}^{\times})$, as this is (up to the homotopy
classification in $\mathsf{Picard}$) the unique connected Picard groupoid with
trivial stable $k$-invariant (Def. \ref{def_Signature}) and automorphism group
$\mathbb{Q}_{>0}^{\times}$ of its tensor unit. Since Theorem \ref{thm1} shows
that the induced symmetric monoidal functor%
\begin{equation}
V(\mathsf{LCA}_{\operatorname*{vf}})\longrightarrow V(\mathsf{LCA})
\label{lv3}%
\end{equation}
on $\pi_{1}$ of the Picard groupoids is just the inclusion $\mathbb{Q}%
_{>0}^{\times}\subset\mathbb{R}_{>0}^{\times}$ and $V(\mathsf{LCA})$
corresponds the usual Haar torsor, it follows that the universal determinant
on $\mathsf{LCA}_{\operatorname*{vf}}$ can itself be interpreted as suitable
choices of Haar measures, namely exactly those which only differ by positive
rational multiples from any fixed initial choice. The symmetric monoidal
functor of Eq. \ref{lv3} (by Example \ref{example_TorsAIsLikeSigmaHA}) after
applying $\Psi$ turns into the map of spectra
\begin{equation}
i\colon\Sigma H\mathbb{Q}_{>0}^{\times}\longrightarrow\Sigma H\mathbb{R}%
_{>0}^{\times} \label{lv5}%
\end{equation}
(or rather a homotopy class of maps of spectra). Hence, in order to prove that%
\[
\mathsf{LCA}_{\operatorname*{vf}}^{\times}\overset{u}{\longrightarrow
}\mathsf{Tors}(\mathbb{Q}_{>0}^{\times})\overset{i_{\ast}}{\longrightarrow
}\mathsf{Tors}(\mathbb{R}_{>0}^{\times})
\]
agrees with $Ha\mid_{\mathsf{LCA}_{\operatorname*{vf}}}$, we just need to show
that $i_{\ast}$ \textit{also} has the property that $\Psi$ sends it to the map
of Eq. \ref{lv5}. But this is just Example
\ref{example_BaseChangeOfTorsAgreesWithNaturalMap}.\newline
\end{proof}

Now we are ready to prove a claim we have made much earlier.

\begin{proof}
[Proof of Theorem \ref{thm_HaQCharacterization}]In
\S \ref{sect_RationalizedHaarMeasure} we introduced a functor $Ha^{\mathbb{Q}%
}\colon\mathsf{LCA}_{\operatorname*{vf}}^{\times}\rightarrow\mathsf{Tors}%
(\mathbb{Q}_{>0}^{\times})$, but we did not supply the extra data needed to
pin down a determinant functor as in Def. \ref{def_DeterminantFunctor}. The
claim we have to prove here amounts to saying that it is possible to extend
$Ha^{\mathbb{Q}}$ to a true determinant functor. We will prove this as
follows:\ We will instead work with Deligne's universal determinant functor
\cite{MR902592}, which we denote by $u$,%
\begin{equation}
\mathsf{LCA}_{\operatorname*{vf}}^{\times}\overset{u}{\longrightarrow
}V(\mathsf{LCA}_{\operatorname*{vf}})\cong\mathsf{Tors}(\mathbb{Q}%
_{>0}^{\times}) \label{lg_zp0}%
\end{equation}
(Prop. \ref{prop_1}). It tautologically satisfies the demands of Def.
\ref{def_DeterminantFunctor} and we shall show, reversely, that on objects and
arrows it can be identified with the description of $Ha^{\mathbb{Q}}$ in
\S \ref{sect_RationalizedHaarMeasure}. To this end, let $\mathsf{LCA}%
_{\operatorname*{vf}}^{\operatorname*{dec}}$ (\textquotedblleft%
\textit{decorated} vector-free LCA\ groups\textquotedblright) be the category
of pairs $(X,C)$ with $X\in\mathsf{LCA}$ and $C$ a compact open in $X$.
Morphisms $(X^{\prime},C^{\prime})\rightarrow(X,C)$ are morphisms $X^{\prime
}\rightarrow X$ of the LCA\ groups and there is no interaction with the
choices of $C^{\prime}$ or $C$. Evidently, the forgetful functor%
\begin{align*}
T\colon\mathsf{LCA}_{\operatorname*{vf}}^{\operatorname*{dec}}  &
\longrightarrow\mathsf{LCA}_{\operatorname*{vf}}\\
(X,C)  &  \longmapsto X
\end{align*}
is an equivalence of categories. Equip $\mathsf{LCA}_{\operatorname*{vf}%
}^{\operatorname*{dec}}$ with the induced exact structure so that $T$ becomes
an exact equivalence of exact categories. A sequence is exact iff $T$ sends it
to an exact sequence. But this means that the universal determinant functor
$u^{\operatorname*{dec}}$ of $\mathsf{LCA}_{\operatorname*{vf}}%
^{\operatorname*{dec}}$ is just the same, or said differently:\ The diagram of
solid arrows%
\begin{equation}%
{
\begin{tikzcd}
	{\mathsf{LCA}_{\operatorname{vf} }^{\operatorname{dec} \times}}
&&&& {\mathsf{LCA}_{\operatorname{vf} }^{\times}} \\
	&& {\mathsf{Tors}({\mathbb{Q}_{>0}^{\times}})} \\
	\\
	&& {\mathsf{Tors}({\mathbb{R}_{>0}^{\times}})}
	\arrow["T", shift left, from=1-1, to=1-5]
	\arrow["{u^{\operatorname{dec} }}"', from=1-1, to=2-3]
	\arrow["{h^{\operatorname{dec} }}"', from=1-1, to=4-3]
	\arrow["H", shift left, dashed, from=1-5, to=1-1]
	\arrow["u", from=1-5, to=2-3]
	\arrow["{Ha\mid_{\mathsf{LCA}_{\operatorname{vf} }}}", from=1-5, to=4-3]
	\arrow["{i_*}", from=2-3, to=4-3]
\end{tikzcd}
}
\label{lg_zp1}%
\end{equation}
commutes by Eq. \ref{lt3}. Now recall the description of $Ha^{\mathbb{Q}}$ in
\S \ref{sect_RationalizedHaarMeasure}: In Step 1, for any object
$X\in\mathsf{LCA}_{\operatorname*{vf}}$ we pick a compact open $C\subseteq X$.
Any such choice can be prolonged to a choice for any object, but that datum is
just what we need to pick a concrete inverse equivalence $H$ (the dashed arrow
in Diagram \ref{lg_zp1}). Now consider the Haar measure, restricted to
$\mathsf{LCA}_{\operatorname*{vf}}$. It similarly admits a lift to the
decorated category, denoted by $h^{\operatorname*{dec}}$ above. We can now
trivialize the torsors $Ha(X)$ for all objects in $\mathsf{LCA}%
_{\operatorname*{vf}}^{\operatorname*{dec}}$: In the torsor of Haar measures
of $X$ we can pick the unique element $\mu_{\operatorname*{root}}\in Ha(X)$
such that $\mu_{\operatorname*{root}}(C)=1$. Then $h^{\operatorname*{dec}%
}(X,C)$ in Diagram \ref{lg_zp1} can be described as the multiples
$\mathbb{R}_{>0}^{\times}\cdot\mu_{\operatorname*{root}}$ inside, and agreeing
with all of, $Ha(X)$. Since this construction of $\mu_{\operatorname*{root}}$
matches the recipe in Step 2 of \S \ref{sect_RationalizedHaarMeasure}, we
precisely get the characterization that%
\begin{equation}%
{
\begin{tikzcd}
	{Ha^{\mathbb{Q} }(X,C)} & \subset& {Ha(X,C)} \\
	{\mathbb{Q}_{>0}^{\times} \cdot\mu_{\operatorname{root} }} & \subset
& {\mathbb{R}_{>0}^{\times} \cdot\mu_{\operatorname{root} }.}
	\arrow[equals, from=1-1, to=2-1]
	\arrow[equals, from=1-3, to=2-3]
\end{tikzcd}
}
\label{d_5}%
\end{equation}
Any arrow $f\colon X^{\prime}\rightarrow X$ in $\mathsf{LCA}%
_{\operatorname*{vf}}^{\times}$ lifts under $H$ (of Diagram \ref{lg_zp1}) to
an arrow $(X^{\prime},C^{\prime})\rightarrow(X,C)$ in $\mathsf{LCA}%
_{\operatorname*{vf}}^{\operatorname*{dec}\times}$ and since the induced map
in the Haar torsor is just basechanged from $\mathbb{Q}$ to $\mathbb{R}$ by%
\[
h^{\operatorname*{dec}}=i_{\ast}\circ u^{\operatorname*{dec}}\text{,}%
\]
the distinguished subgroups of rational multiples, as in Eq. \ref{d_5}, are
respected by $h^{\operatorname*{dec}}$. This finishes the proof.
\end{proof}

\begin{theorem}
\label{thm_main}Suppose $(\mathsf{P},\boxtimes)$ is a Picard groupoid and%
\[
\mathcal{D}\colon\mathsf{LCA}_{\operatorname*{vf}}^{\times}\longrightarrow
\mathsf{P}%
\]
is any determinant functor. Then there exists a morphism of Picard groupoids
$f$ such that%
\begin{equation}%
{
\begin{tikzcd}
	{\mathsf{LCA}_{\operatorname{vf} }^{\times}} && {(\mathsf{Tors}(\mathbb
{Q}_{>0}^{\times}),\otimes)} \\
	\\
	&& {(\mathsf{P},\boxtimes)}
	\arrow["{{Ha^\mathbb{Q} }}", from=1-1, to=1-3]
	\arrow["{{\mathcal{D}}}"', from=1-1, to=3-3]
	\arrow["f", dashed, from=1-3, to=3-3]
\end{tikzcd}
}
\label{diag_p_2}%
\end{equation}
commutes, where $Ha^{\mathbb{Q}}$ is the Haar measure determinant functor,
restricted to only allowing rational multiples
(\S \ref{sect_RationalizedHaarMeasure}). And more precisely, $Ha^{\mathbb{Q}}$
is the universal determinant functor of $\mathsf{LCA}_{\operatorname*{vf}}$ in
the sense of Def. \ref{def_UnivDetFunctor}.
\end{theorem}

\begin{proof}
The functor $Ha^{\mathbb{Q}}$ of \S \ref{sect_RationalizedHaarMeasure} extends
by Theorem \ref{thm_HaQCharacterization} to a\footnote{or: the} universal
determinant functor on $\mathsf{LCA}_{\operatorname*{vf}}$. The factorization
in our claim, Diagram \ref{diag_p_2}, therefore follows from the
characterization of universality in Def. \ref{def_UnivDetFunctor}.
\end{proof}

%

\appendix

\section{Gillet--Grayson model\label{sect_GilletGraysonModel}}

Let $\mathsf{C}$ be a pointed exact category, i.e., an exact category with a
fixed choice of a zero object. This will be denoted by $0$. Following
\cite{MR909784, MR2007234}, define a simplicial set $G_{\bullet}\mathsf{C}$
whose $n$-simplices are given by a pair of commutative diagrams%
\[%
\xymatrix@!=0.157in{
&                         &                         &                        & P_{n/(n-1)}
\\
&                         &                         & \cdots\ar@{^{(}%
.>}[r] & \vdots\ar@{.>>}[u] \\
&                         & P_{2/1} \ar@{^{(}.>}[r] & \cdots\ar@{^{(}%
.>}[r] & P_{n/1} \ar@{.>>}[u] \\
& P_{1/0} \ar@{^{(}.>}[r] & P_{2/0} \ar@{^{(}.>}[r] \ar@{.>>}[u] & \cdots
\ar@{^{(}.>}[r] & P_{n/0} \ar@{.>>}[u] \\
P_0 \ar@{^{(}.>}[r] & P_1 \ar@{^{(}.>}[r] \ar@{.>>}[u] & P_2 \ar@{^{(}%
.>}[r] \ar@{.>>}[u] & \cdots\ar@{^{(}.>}[r] & P_n \ar@{.>>}[u]
}%
\qquad%
\xymatrix@!=0.157in{
&                         &                         &                        & P_{n/(n-1)}
\\
&                         &                         & \cdots\ar@{^{(}%
->}[r] & \vdots\ar@{->>}[u] \\
&                         & P_{2/1} \ar@{^{(}->}[r] & \cdots\ar@{^{(}%
->}[r] & P_{n/1} \ar@{->>}[u] \\
& P_{1/0} \ar@{^{(}->}[r] & P_{2/0} \ar@{^{(}->}[r] \ar@{->>}[u] & \cdots
\ar@{^{(}->}[r] & P_{n/0} \ar@{->>}[u] \\
P^{\prime}_0 \ar@{^{(}->}[r] & P^{\prime}_1 \ar@{^{(}->}[r] \ar@
{->>}[u] & P^{\prime}_2 \ar@{^{(}->}[r] \ar@{->>}[u] & \cdots\ar@{^{(}%
->}[r] & P^{\prime}_n \ar@{->>}[u]
}%
\text{,}%
\]
such that (1) the diagrams agree strictly\footnote{i.e., not just up to a
natural isomorphism.} above the bottom row, (2) each sequence $P_{i}%
\hookrightarrow P_{j}\twoheadrightarrow P_{j/i}$ is exact, (2') each sequence
$P_{i}^{\prime}\hookrightarrow P_{j}^{\prime}\twoheadrightarrow P_{j/i}%
^{\prime}$ is exact, (3) each sequence $P_{i/j}\hookrightarrow P_{m/j}%
\twoheadrightarrow P_{m/i}$ is exact. The face and degeneracy maps come from
deleting the $i$-th row and column, resp. by duplicating them. For details we
refer to the references. The $0$-simplices are pairs $(P,P^{\prime})$ of
objects. The $1$-simplices are pairs of exact sequences%
\begin{equation}%
\xymatrix{
P_0 \ar@{^{(}.>}[r] & P_1 \ar@{.>>}[r] & P_{1/0} & \qquad& P^{\prime}%
_0 \ar@{^{(}->}[r] & P^{\prime}_1 \ar@{->>}[r] & P_{1/0}
}
\label{lcimez22}%
\end{equation}
with the same cokernel\footnote{i.e., not just up to a natural isomorphism.}.
This pair corresponds to a $1$-simplex from the point $(P_{0},P_{0}^{\prime})$
to the point $(P_{1},P_{1}^{\prime})$. The main result of Gillet and Grayson
is the equivalence%
\[
K^{\operatorname*{conn}}(\mathsf{C})\cong\left\vert G_{\bullet}\mathsf{C}%
\right\vert \text{,}%
\]
or more specifically:\ They equip the space $\left\vert G_{\bullet}%
\mathsf{C}\right\vert $ with an infinite loop space structure and identifying
it with a connective spectrum, it is a model for $K^{\operatorname*{conn}%
}(\mathsf{C})$.

\begin{example}
\label{example_pi0GG}The identification with the zero-th $K$-group is as
follows: the $0$-simplex $(P,P^{\prime})$ lies in the connected component
$[P^{\prime}]-[P]\in\pi_{0}K^{\operatorname*{conn}}(\mathsf{C})$. Other
authors use other sign conventions.\footnote{Weibel's $K$-book uses precisely
the opposite signs.}
\end{example}

\begin{example}
\label{example_pi1GG}The identification with the first $K$-group is more
complicated. We only need to know that any automorphism $P\overset{\varphi
}{\longrightarrow}P$ of an object $P\in\mathsf{C}$ determines a unique class
in $\pi_{1}K^{\operatorname*{conn}}(\mathsf{C})$, corresponding to the loop%
\begin{equation}%
{
\begin{tikzcd}
	{(0,0)} &&&& {(P,P)} \\
	\\
	&& {(0,0)}
	\arrow["\xi", from=1-1, to=1-5]
	\arrow["{\nu(0)}", from=3-3, to=1-1]
	\arrow["{\nu(P)}"', from=3-3, to=1-5]
\end{tikzcd}
}
\label{diag_loop}%
\end{equation}
around the basepoint $(0,0)$ (where $0$ is the designated zero object of the
pointed category $\mathsf{C}$), where $\nu(P)$ and $\xi$ come from the
$1$-simplices%
\begin{equation}%
{
\begin{tikzcd}
	0 && P && P \\
	0 && P && P
	\arrow[from=1-1, to=1-3]
	\arrow["1", from=1-3, to=1-5]
	\arrow[dotted, from=2-1, to=2-3]
	\arrow["1"', dotted, from=2-3, to=2-5]
\end{tikzcd}
}%
\qquad\text{and}\qquad%
{
\begin{tikzcd}
	0 && P && P \\
	0 && P && P
	\arrow[from=1-1, to=1-3]
	\arrow["\varphi", from=1-3, to=1-5]
	\arrow[dotted, from=2-1, to=2-3]
	\arrow["1", dotted, from=2-3, to=2-5]
\end{tikzcd}
}
\label{lcz60}%
\end{equation}
in $G_{\bullet}(\mathsf{C})$ respectively. See \cite{MR1409623} for details
and proofs.
\end{example}

\begin{remark}
It follows from Prop. \ref{prop_HomotopyCatsOfPicardGroupoidsAndSpectra} that
the $1$-truncation $\tau_{\leq1}K^{\operatorname*{conn}}(\mathsf{C})$ can, up
to stable homotopy type, be identified with a Picard groupoid. Deligne's work
\cite{MR902592} shows that the truncation map\footnote{the co-unit of the
adjoint pair attached to the inclusion of $1$-truncated spectra into all
spectra} $K^{\operatorname*{conn}}(\mathsf{C})\longrightarrow\tau_{\leq
1}K^{\operatorname*{conn}}(\mathsf{C})$ essentially can be identified with the
concept of a determinant functor. This entire text rests on making this idea
explicit for $\mathsf{LCA}_{\operatorname*{vf}}$.
\end{remark}

\begin{acknowledgement}
We thank M. Groechenig and D. Macias Castillo for their help.
\end{acknowledgement}

\bibliographystyle{amsalpha}
\bibliography{ollinewbib}

\def\cprime{$'$} \def\polhk#1{\setbox0=\hbox{#1}{\ooalign{\hidewidth
  \lower1.5ex\hbox{`}\hidewidth\crcr\unhbox0}}} \def\cprime{$'$}
  \def\cprime{$'$} \def\cprime{$'$} \def\cprime{$'$}
\providecommand{\bysame}{\leavevmode\hbox to3em{\hrulefill}\thinspace}
\providecommand{\MR}{\relax\ifhmode\unskip\space\fi MR }
\providecommand{\MRhref}[2]{%
  \href{http://www.ams.org/mathscinet-getitem?mr=#1}{#2}
}
\providecommand{\href}[2]{#2}
\begin{thebibliography}{ADGB22}

\bibitem[ADGB22]{MR4510389}
L.~Au{\ss }enhofer, D.~Dikranjan, and A.~Giordano~Bruno, \emph{Topological
  groups and the {P}ontryagin--van {K}ampen duality---an introduction}, De
  Gruyter Studies in Mathematics, vol.~83, De Gruyter, Berlin, [2022]
  \copyright 2022. \MR{4510389}

\bibitem[Art24]{MR4831262}
M.~Artusa, \emph{Duality for condensed cohomology of the {W}eil group of a
  {$p$}-adic field}, Doc. Math. \textbf{29} (2024), no.~6, 1381--1434.
  \MR{4831262}

\bibitem[Art25]{artusa2025}
M.~Artusa, \emph{{D}uality for the condensed {W}eil-\'etale realisation of
  $1$-motives over $p$-adic fields}, 2025.

\bibitem[BK90]{MR1086888}
S.~Bloch and K.~Kato, \emph{{$L$}-functions and {T}amagawa numbers of motives},
  The {G}rothendieck {F}estschrift, {V}ol.\ {I}, Progr. Math., vol.~86,
  Birkh\"auser Boston, Boston, MA, 1990, pp.~333--400. \MR{1086888}

\bibitem[Bra19]{MR4028830}
O.~Braunling, \emph{On the relative {$K$}-group in the {ETNC}}, New York J.
  Math. \textbf{25} (2019), 1112--1177. \MR{4028830}

\bibitem[Bre11]{MR2842932}
M.~Breuning, \emph{Determinant functors on triangulated categories}, J.
  K-Theory \textbf{8} (2011), no.~2, 251--291. \MR{2842932}

\bibitem[B{\"u}h10]{MR2606234}
T.~B{\"u}hler, \emph{Exact categories}, Expo. Math. \textbf{28} (2010), no.~1,
  1--69. \MR{2606234 (2011e:18020)}

\bibitem[Cla17]{clausen}
D.~Clausen, \emph{A {K}-theoretic approach to {A}rtin maps}, arXiv:1703.07842
  [math.KT] (2017).

\bibitem[Del87]{MR902592}
P.~Deligne, \emph{Le d\'eterminant de la cohomologie}, Current trends in
  arithmetical algebraic geometry ({A}rcata, {C}alif., 1985), Contemp. Math.,
  vol.~67, Amer. Math. Soc., Providence, RI, 1987, pp.~93--177. \MR{902592
  (89b:32038)}

\bibitem[FM18]{MR3874942}
M.~Flach and B.~Morin, \emph{Weil-\'{e}tale cohomology and zeta-values of
  proper regular arithmetic schemes}, Doc. Math. \textbf{23} (2018),
  1425--1560. \MR{3874942}

\bibitem[Fol16]{MR3444405}
G.~B. Folland, \emph{A course in abstract harmonic analysis}, second ed.,
  Textbooks in Mathematics, CRC Press, Boca Raton, FL, 2016. \MR{3444405}

\bibitem[FPR94]{MR1265546}
J.-M. Fontaine and B.~Perrin-Riou, \emph{Autour des conjectures de {B}loch et
  {K}ato: cohomologie galoisienne et valeurs de fonctions {$L$}}, Motives
  ({S}eattle, {WA}, 1991), Proc. Sympos. Pure Math., vol.~55, Amer. Math. Soc.,
  Providence, RI, 1994, pp.~599--706. \MR{1265546}

\bibitem[GG87]{MR909784}
H.~Gillet and D.~Grayson, \emph{The loop space of the {$Q$}-construction},
  Illinois J. Math. \textbf{31} (1987), no.~4, 574--597. \MR{909784}

\bibitem[GG03]{MR2007234}
\bysame, \emph{Erratum to: ``{T}he loop space of the {$Q$}-construction''},
  Illinois J. Math. \textbf{47} (2003), no.~3, 745--748. \MR{2007234}

\bibitem[GM24]{MR4699875}
T.~Geisser and B.~Morin, \emph{Pontryagin duality for varieties over {$p$}-adic
  fields}, J. Inst. Math. Jussieu \textbf{23} (2024), no.~1, 425--462.
  \MR{4699875}

\bibitem[Hv19]{hr2}
R.~Henrard and A.-C. {van Roosmalen}, \emph{Derived categories of one-sided
  exact categories and their localizations}, \texttt{arXiv:1903.12647} (2019).

\bibitem[JO12]{MR2981952}
N.~Johnson and A.~Osorno, \emph{Modeling stable one-types}, Theory Appl. Categ.
  \textbf{26} (2012), No. 20, 520--537. \MR{2981952}

\bibitem[Knu02]{MR1914072}
F.~Knudsen, \emph{Determinant functors on exact categories and their extensions
  to categories of bounded complexes}, Michigan Math. J. \textbf{50} (2002),
  no.~2, 407--444. \MR{1914072}

\bibitem[KS99]{MR1687096}
R.~Kottwitz and D.~Shelstad, \emph{Foundations of twisted endoscopy},
  Ast\'{e}risque (1999), no.~255, vi+190. \MR{1687096}

\bibitem[KS06]{MR2182076}
M.~Kashiwara and P.~Schapira, \emph{Categories and sheaves}, Grundlehren der
  Mathematischen Wissenschaften [Fundamental Principles of Mathematical
  Sciences], vol. 332, Springer-Verlag, Berlin, 2006. \MR{2182076}

\bibitem[Lic09]{MR2552104}
S.~Lichtenbaum, \emph{The {W}eil-\'{e}tale topology for number rings}, Ann. of
  Math. (2) \textbf{170} (2009), no.~2, 657--683. \MR{2552104}

\bibitem[Lic24]{MR4699877}
\bysame, \emph{Special values of zeta-functions of regular schemes}, J. Inst.
  Math. Jussieu \textbf{23} (2024), no.~1, 495--519. \MR{4699877}

\bibitem[Loo53]{MR54173}
L.~Loomis, \emph{An introduction to abstract harmonic analysis}, D. Van
  Nostrand Co., Inc., Toronto-New York-London, 1953. \MR{54173}

\bibitem[MTW15]{MR3302579}
F.~Muro, A.~Tonks, and M.~Witte, \emph{On determinant functors and
  {$K$}-theory}, Publ. Mat. \textbf{59} (2015), no.~1, 137--233. \MR{3302579}

\bibitem[Nen96]{MR1409623}
A.~Nenashev, \emph{Double short exact sequences produce all elements of
  {Q}uillen's {$K_1$}}, Algebraic {$K$}-theory ({P}ozna\'n, 1995), Contemp.
  Math., vol. 199, Amer. Math. Soc., Providence, RI, 1996, pp.~151--160.
  \MR{1409623}

\bibitem[Oes83]{MR750319}
J.~Oesterl\'{e}, \emph{Compatibilit\'{e} de la suite exacte de {P}oitou-{T}ate
  aux mesures de {H}aar}, Seminar on number theory, 1982--1983 ({T}alence,
  1982/1983), Univ. Bordeaux I, Talence, 1983, pp.~Exp. No. 19, 17. \MR{750319}

\bibitem[Pat12]{MR2981817}
D.~Patel, \emph{de {R}ham epsilon factors}, Invent. Math. \textbf{190} (2012),
  no.~2, 299--355. \MR{2981817}

\bibitem[Sch06]{MR2206639}
M.~Schlichting, \emph{Negative {$K$}-theory of derived categories}, Math. Z.
  \textbf{253} (2006), no.~1, 97--134. \MR{2206639}

\bibitem[She96]{MR1409625}
C.~Sherman, \emph{Connecting homomorphisms in localization sequences},
  Algebraic {$K$}-theory ({P}ozna\'n, 1995), Contemp. Math., vol. 199, Amer.
  Math. Soc., Providence, RI, 1996, pp.~175--183. \MR{1409625}

\bibitem[She98]{MR1621689}
\bysame, \emph{On {$K_1$} of an exact category}, $K$-Theory \textbf{14} (1998),
  no.~1, 1--22. \MR{1621689}

\bibitem[SR72]{MR0338002}
N.~Saavedra~Rivano, \emph{Cat\'egories {T}annakiennes}, Lecture Notes in
  Mathematics, Vol. 265, Springer-Verlag, Berlin-New York, 1972. \MR{0338002}

\bibitem[Wei13]{MR3076731}
C.~Weibel, \emph{The {$K$}-book}, Graduate Studies in Mathematics, vol. 145,
  American Mathematical Society, Providence, RI, 2013, An introduction to
  algebraic $K$-theory. \MR{3076731}

\end{thebibliography}

\end{document}